\documentclass[11pt,reqno]{amsart}
\usepackage{fullpage}

\usepackage{graphicx}
\usepackage[draft]{hyperref}
\usepackage{amsmath,amsopn,amssymb,amsfonts,stmaryrd}
\usepackage{verbatim}
\usepackage{amsthm}
\usepackage{mathtools}
\usepackage{color}
\usepackage{enumitem}
\usepackage[framemethod=TikZ]{mdframed}
\usepackage{bbm}
\usepackage{mathrsfs}
\usepackage{booktabs}
\usepackage{caption}
\usepackage{cleveref}
\usepackage{bm}
\usepackage{makecell}
\usepackage{ latexsym }
\usepackage{cases}

\theoremstyle{plain}
\newtheorem*{remark}{Remark}

\newtheorem{theorem}{Theorem}[section]
\newtheorem{corollary}[theorem]{Corollary}

\newtheorem{lemma}[theorem]{Lemma}

\theoremstyle{definition}

\theoremstyle{remark}

\newcommand{\Z}{{\mathbb Z}}
\newcommand{\N}{{\mathbb N}}
\newcommand{\R}{\mathbb{R}}

\newcommand{\C}{\mathbb{C}}
\newcommand{\E}{\mathbb{E}}
\renewcommand{\H}{\mathbb{H}}

\newcommand{\calC}{\mathcal{C}}

\newcommand{\calE}{\mathcal{E}}

\newcommand{\calK}{\mathcal{K}}

\newcommand{\calU}{\mathcal{U}}
\newcommand{\calV}{\mathcal{V}}

\newcommand{\abcd}{\begin{pmatrix}a&b\\c&d\end{pmatrix}}
\newcommand{\pabcd}{\left(\begin{smallmatrix}a&b\\c&d\end{smallmatrix}\right)}

\renewcommand{\Im}{{\rm Im}}

\newcommand{\sgn}{\mbox{sgn}}

\newcommand{\SL}{\mathrm{SL}}

\newcommand{\Arg}{\mathrm{Arg}}

\newcommand{\ol}[1]{\overline{#1}}

\newcommand{\zz}{\mathfrak{z}}

\newcommand{\pd}[2]{\frac{\partial^#1}{\partial #2^#1}}

\numberwithin{equation}{section}

\newsavebox{\mycases}

\allowdisplaybreaks

\newcommand{\pmat}[1]{\left( \smallmatrix #1 \endsmallmatrix \right)}
\newcommand{\mat}[1]{\left( \begin{matrix} #1 \end{matrix} \right)}

\newcommand{\flo}[1]{\lfloor #1\rfloor}

\renewcommand{\sgn}{\textnormal{sgn}}

\def\a{\alpha}
\def\b{\beta}
\def\d{\delta}

\def\k{\kappa}
\def\l{\lambda}

\def\vp{\varrho}

\def\z{\zeta}

\def\vth{\vartheta}

\def\e{\varepsilon}

\def\g{\gamma}
\def\t{\tau}

\def\TH{\Theta}

\def\del{  \partial}

\newcommand{\rank}{\operatorname{rank}}

\newcommand{\coeff}{\operatorname{coeff}}

\newcommand{\re}{{\rm Re}}
\newcommand{\im}{{\rm Im}}

\renewcommand{\sgn}{{\rm sgn}}

\def\wh{\widehat}

\setlist[itemize]{noitemsep, topsep=0pt}

\makeatletter
\newcommand{\vast}{\bBigg@{2}}
\newcommand{\Vast}{\bBigg@{5}}
\makeatother

\renewcommand{\pmod}[1]{\   \left(  \mathrm{mod} \,  #1 \right)}
\newcommand{\Pmod}[1]{\    (  \mathrm{mod} \,  #1 )}
\newcommand{\erf}{{\mathrm{erf}}}

\title{Taylor coefficients of false Jacobi forms and ranks of unimodal sequences}

\thanks{This project has received funding from the European Research Council (ERC) under the European Union's Horizon 2020 research and innovation programme (grant agreement No. 101001179).  The authors were partially supported by the supported by the SFB/TRR 191 ``Symplectic Structures in Geometry, Algebra and Dynamics'', funded by the DFG (Projektnummer 281071066 TRR 191).}


\author{Walter Bridges and Kathrin Bringmann}


\begin{document}
\maketitle

\begin{abstract}
	 We apply the new framework for modularity of false theta functions developed by the second author and Nazaroglu to study the asymptotic behavior of Taylor coefficients of false Jacobi forms. The examples we study generate moments of the rank for unimodal sequences. For two types of unimodal sequences, we prove asymptotic series for the rank moments.
\end{abstract}

\section{Introduction and statement of results}\label{sec:intro}

In the theory of partitions, there are many generating functions $F(\z;q)$ of combinatorial interest that factor as 
\begin{equation}\label{E:genfnfactor}
	F(\zeta;q)=\text{Jacobi form} \cdot  \text{partial/false Jacobi theta function}  + \text{sparse series}.
\end{equation}
Setting $\z=e^{2\pi iz}$, we wish to study Taylor coefficients in $z$ of such $F(\z;q)$, themselves $q$-series, and we would like to describe the growth of the coefficients of $q^n$ as $n\to\infty$. That is, we would like a method to estimate
\[
	\coeff_{\left[q^n\right]} \frac{1}{(2\pi i)^\ell}\left[\frac{\del^\ell}{\del z^\ell} F(\z;q)\right]_{z=0},\qquad \text{as }n\to\infty.
\]
As modular (and mock modular) Jacobi forms are well understood, the primary difficulty is to understand the contribution from the partial/false theta function. In the present work, we extend the recent framework for the modularity of false theta functions due to the second author and Nazaroglu \cite{CN} to accommodate two examples. We prove two variable transformations for false theta functions, and we use the Hardy--Ramanujan Circle Method to  obtain asymptotic series.

\subsection{False theta functions}

{\it False theta functions} appear similar to classical theta functions but have different sign factors that prevent them from being modular forms. Consider in particular the false theta function\footnote{Note that in \cite{CN} a slightly different definition of $\psi$ was employed, where $\sgn(m)$ was used instead of $\sgn(m+\frac{z_2}{\t_2})$. Note that for $|\frac{z_2}{\t_2}|<\frac12$ these definitions agree.} ($\t=\t_1+i\t_2\in\H$, $z=z_1+iz_2\in\C$) 
\[
	\psi(z;\t) := i\sum_{m\in\Z+\frac12} \sgn\left(m+\frac{z_2}{\t_2}\right)(-1)^{m-\frac12}q^\frac{m^2}{2}\z^m \qquad \left(q:=e^{2\pi i\t},\z:=e^{2\pi iz},w\in\H\right).
\]
In recent work \cite{CN} the second author and Nazaroglu defined the following completion of $\psi$,
\[
	\wh\psi(z;\t,w) := i\sum_{m\in\Z+\frac12} \erf\left(-i\sqrt{\pi i(w-\t)}\left(m+\frac{z_2}{\t_2}\right)\right) (-1)^{m-\frac12} q^\frac{m^2}{2} \z^m,
\]
where $\erf(x):=\frac{2}{\sqrt\pi}\int_0^xe^{-t^2}dt$ is the {\it error function}, and where here and throughout we take the principal branch of the square-root. That $\wh\psi$ is a ``completion'' of $\psi$ means that $\psi$ occurs as limiting value of $\wh\psi$ (see \eqref{limit}) and $\wh\psi$ transforms like a Jacobi form.

\begin{theorem}[\cite{CN}, Theorem 2.3]\label{T:CNtrans}
	Let $z\in\C$ and $\t,w\in\H$. Then for $\pabcd\in\SL_2(\Z)$, we have
	\[
		\wh\psi(z;\t,w) = \e_{\t,w}\abcd^{-1}\chi\abcd^{-3} (c\t+d)^{-\frac12} e^{-\frac{\pi icz^2}{c\t+d}} \wh\psi\left(\frac{z}{c\t+d};\frac{a\t+b}{c\t+d},\frac{aw+b}{cw+d}\right),
	\]
	where $\e_{\t,w}\pabcd$ is defined in \eqref{E:epsilonbranchcutfactor} and $\chi$ is given in \eqref{E:etamultiplier}.
\end{theorem}

From the modularity of $\widehat{\psi}$ we derive a transformation law for $\psi$ requiring the Eichler integrals
\begin{equation*}\label{E:MordellEdefinition}
	\calE_\frac ac(z;\t) := e^{-\frac{\pi iaz_2^2}{c\t_2^2}}\int_\frac ac^{\t+i\infty+\e} e^\frac{\pi iz_2^2\zz}{\t_2^2} \frac{\sum_{m\in\Z+\frac12} (-1)^{m-\frac12}\left(m+\frac{z_2}{\t_2}\right)e^{\pi i\left(m^2\zz+2m\left(z+(\zz-\t)\frac{z_2}{\t_2}\right)\right)}}{\sqrt{i(\zz-\t)}} d\zz.
\end{equation*}
The following theorem extends Lemma 3.1 of \cite{CN} which treats the case $z=0$.
\begin{theorem}\label{T:psitrans}
	For $z\in\C$, we have, for $\pmat{a&b\\c&d}\in\SL_2(\Z)$ with $c>0$,
	\begin{multline*}
		\psi(z;\t) = \chi\abcd^{-3} (c\t+d)^{-\frac12} e^{-\frac{\pi icz^2}{c\t+d}}\\
		\times \left(\psi\left(\frac{z}{c\t+d};\frac{a\t+b}{c\t+d}\right) - ie^{\frac{\pi i}{c(c\t+d)}\left(\frac{\im\left(\frac{z}{c\t+d}\right)}{\im\left(\frac{a\t+ d}{c\t+d}\right)}\right)^2} \calE_\frac ac\left(\frac{z}{c\t+d};\frac{a\t+b}{c\t+d}\right)\right).
	\end{multline*}
\end{theorem}

Theorem \ref{T:psitrans} allows to estimate the Taylor coefficients in $z$ of $\psi$ if $\t$ tends to a rational number.

\subsection{Partition statistics: ranks and cranks}  

A {\it partition} $\l$ of a positive integer $n$ is a sequence of integers, $\l_1\ge\ldots\ge\l_\ell\ge1$, such that $\l_1+\ldots+\l_\ell=n$. The generating function for $p(n)$, the number of partitions of $n$, is the well-known infinite product (setting $p(0):=1$)
$$
	\sum_{n \geq 0} p(n)q^n = \frac{1}{(q;q)_{\infty}},
$$
where $(a;q)_n:=\prod_{j=0}^{n-1}(1-aq^j)$ for $n\in\N_0\cup\{\infty\}$. Note that $\eta(\t):=q^\frac{1}{24}\prod_{n=1}^\infty(1-q^n)$ is {\it Dedekind's eta function}, a modular form. Groundbreaking work of Hardy and Ramanujan developed the Circle Method to exploit the modularity to obtain an asymptotic series for $p(n)$ \cite{HR}. Rademacher later improved this to the following exact formula \cite{R}. For $h,k\in\Z$ with $\gcd(h,k)=1$, we define $0\le[-h]_k<k$ by $-h[-h]_k\equiv1\Pmod{k}$ and let $\chi$ be the multiplier for the Dedekind eta function as in \eqref{E:etamultiplier}. Define the {\it Kloostermann sums}
$$
	A_k(n):=i^{\frac{1}{2}}\sum_{\substack{0 \leq h < k \\ \gcd(h,k)=1}} \chi\mat{[-h]_k&-\frac{h[-h]_k+1}{k}\\k&-h}e^{-\frac{2\pi i}{24k}\left((24n-1)h+[-h]_k\right)}.
$$
Then
$$
	p(n)=\frac{2\pi}{(24n-1)^{\frac{3}{4}}}\sum_{k \geq 1}\frac{A_k(n)}{k}I_{\frac{3}{2}}\left(\frac{\pi\sqrt{24n-1}}{6k}\right).
$$

On the other hand, Ramanujan's famous congruences show surprising arithmetic regularity in $p(n)$, namely \cite{Ramanujan}
$$
	p(5n+4) \equiv 0 \pmod{5}, \qquad p(7n+5) \equiv 0 \pmod{7}, \qquad p(11n+6) \equiv 0 \pmod{11}.
$$
In attempt to explain these, Dyson defined the {\it rank} of a partition as $\rank(\l):=\l_1-\ell$, and he conjectured that this simple statistic can be used to decompose the set of partitions of $5n+4$ into 5 equal classes \cite{Dyson} and that the same should hold for the congruences modulo 7. These conjectures were settled by Atkin and Swinnerton-Dyer in \cite{ASD}. As the congruence modulo $11$ was yet to be explained, Dyson famously called for a ``crank'' statistic that would dissect all three congruences. Such a statistic was found by Andrews and Garvan \cite{AndrewsGarvan,Garvan}. Define $o(\l)$ as the number of ones in $\l$ and let $\mu(\l)$ be the number of parts larger than $o(\l)$. Then the {\it crank} is defined as
$$
	\mathrm{crank}(\lambda):=\begin{cases} \lambda_1 & \text{if $o(\lambda)=0$,} \\ \mu(\lambda)-o(\lambda) & \text{if $o(\lambda) \geq 1$.} \end{cases}
$$
Let $N(m,n)$ (resp. $M(m,n)$) denote the number of partitions of $n$ with rank (resp. crank) equal to $m$.  Then one has
\begin{align}\label{E:rankgenfn}
	R(\z;q) &:= \sum_{\substack{n\ge0\\m\in\Z}} N(m,n)q^n\z^m = \sum_{n\ge0} \frac{q^{n^2}}{(\z q;q)_n\left(\z^{-1}q;q\right)_n},\\
	\label{E:crankgenfn}
	C(\z;q) &:= \sum_{\substack{n\ge0\\m\in\Z}} M(m,n)\z^mq^n = \frac{(q;q)_\infty}{(\z q;q)_\infty\left(\z^{-1}q;q\right)_\infty}.
\end{align}
Like the partition generating function, these functions also enjoy certain modularity properties. Namely $C(\z;q)$ is essentially a Jacobi form, while $R(\z;q)$ is known to be a mock Jacobi form.

A study of moments of the rank and crank statistics was undertaken by Atkin and Garvan \cite{AG},
\begin{align*}
	\sum_{n\ge0} N_\ell(n)q^n &:= \sum_{n\ge0} \left(\sum_{m\in\Z} m^\ell N(m,n)\right) q^n = \frac{1}{(2\pi i)^\ell}\left[\frac{\del^\ell}{\del z^\ell} R(\z;q)\right]_{z=0},\\
	\sum_{n\ge0} M_\ell(n)q^n &:= \sum_{n\ge0} \left(\sum_{m\in\Z} m^\ell M(m,n)\right) q^n = \frac{1}{(2\pi i)^\ell}\left[\frac{\del^\ell}{\del z^\ell} C(\z;q)\right]_{z=0}.
\end{align*}
The symmetry $\z\mapsto\z^{-1}$ in \eqref{E:rankgenfn} and \eqref{E:crankgenfn} implies that $N_\ell(n)=M_\ell(n)=0$ for $\ell$ odd. Atkin and Garvan's ``rank-crank PDE'' connects $R(\z;q)$ and $C(\z;q)$ and allowed them to prove many identities for these moments. The rank-crank PDE was also exploited by the second author, Mahlburg, and Rhoades \cite{BMR} in their proof of the asymptotic formulas,
\begin{equation}\label{E:rankcrankmomentsasymp}
	N_{2\ell}(n) \sim M_{2\ell}(n) \sim \left(2^{2\ell}-2 \right)|B_{2\ell}| (6n)^{\ell}p(n),
\end{equation}
where $B_k$ is the $k$-th Bernoulli number. Note that for the crank moments, it is straightforward from \eqref{E:crankgenfn} to compute the derivatives of $C(\z;q)$ and write them recursively in terms of Eisenstein series \cite[equation (4.6)]{AG}, whereas derivatives of $R(\z;q)$ are less tractable. In follow-up work, the second author, Mahlburg, and Rhoades significantly improved on \eqref{E:rankcrankmomentsasymp} by exploiting the Jacobi and mock Jacobi transformations of $C(\z;q)$ and $R(\z;q)$ \cite{BMR2}. In particular, they avoided the rank-crank PDE and proved asymptotic series for the moments with polynomial error.

\subsection{Ranks for unimodal sequences}

Let $u(m,n)$ denote the number of unimodal sequences
\begin{equation}\label{E:unimodaldef}
	a_1 \le \ldots \le a_r \le c \ge b_1 \ge \ldots \ge b_s, \qquad\text{where}\qquad s - r = m \quad \text{and} \quad \ a_1+a_2+ \dots + b_s=n.
\end{equation}
The {\it size} of a unimodal sequence is the sum of the parts and the rank is $m$. Early work on enumeration and asymptotics for unimodal sequences is due to Auluck \cite{Auluck} and Wright \cite{Wright2}. Standard techniques (see \cite[\S 2.5]{Stanley}) yield that the generating function for $u(m,n)$ equals
$$
	U(\zeta;q)=\sum_{\substack{n \geq 0 \\ m \in \mathbb{Z}}}u(m,n)q^n\zeta^m=\sum_{n \geq 0} \frac{q^n}{(\zeta q,\zeta^{-1}q;q)_{n}}.
$$
In particular, the enumeration function $u(n)$ for all unimodal sequences can be factored as the product (essentially) of a weakly holomorphic modular form and a false theta function (see \cite{Wright2}),
$$
	U(1;q)=\sum_{n \geq 0} u(n)q^n=\frac{1}{(q;q)_{\infty}^2}\sum_{n \geq 0} (-1)^{n}q^{\frac{n(n+1)}{2}}.
$$
In Theorem \ref{thm:u1}, we show that, after subtracting a sparse series, $U(\zeta;q)$ transforms essentially as a ``false Jacobi form''.

In general, there are barriers preventing natural extensions of partition results to unimodal sequences. A recent achievement was an exact formula for $u(n)$ proved by the second author and Nazaroglu using their framework for modularity of false theta functions. Prior to \cite{CN}, essentially only the main exponential term was known (see \cite{Auluck, Wright2}). The asymptotic formula below corresponds to the second term in \cite[Theorem 1.3]{CN}.\footnote{Note that our definition here differs from the one in \cite{CN}. If one denotes those unimodal sequences by $u^*(n)$, then
\[
	u(n) = \coeff_{\left[q^n\right]} \frac{1}{(q;q)_\infty^2} - u^*(n).
\]} Define the Kloostermann sums $K_k(n,\nu)$ as in \eqref{E:KloostermannUnrestricted}.

\begin{theorem}[\cite{CN}, Theorem 1.3]\label{thm:cn1.3}
	We have
	\begin{multline*}
		u(n) = \frac{\pi}{2^\frac34\sqrt{3}(24n+1)^\frac34}\sum_{k\ge1} \sum_{0\le\nu\le2k-1} \frac{K_k(n,\nu)}{k^2}\\
		\times \int_{-1}^1 \cot\left(\frac{\pi}{2k}\left(\frac{x}{\sqrt{6}}-\nu-\frac12\right)\right) I_\frac32\left(\frac{\pi}{3\sqrt{2}k}\sqrt{\left(1-x^2\right)(24n+1)}\right)\left(1-x^2\right)^\frac34 dx.
	\end{multline*}
\end{theorem}

In the present paper, we study the $\ell$-th unimodal rank moments, defined by
\[
	u_\ell(n) := \sum_{m\in\Z} m^\ell u(m,n) = \coeff_{\left[q^n\right]}\left[\left(\z\frac{\del}{\del\z}\right)^\ell U(\z;q)\right]_{\z=1}.
\]
For $\ell$ odd we have $u_\ell(n)=0$ by symmetry. For $\ell$ even the main asymptotic term and resulting logistic distribution of the rank was proved by the second author, Jennings-Shaffer, and Mahlburg \cite{BJSM}. We greatly improve their results to an asymptotic series, given in terms of integrals involving
$$
	C_j(w):=\left(\frac{1}{2\pi i}\frac{\del}{\del w}\right)^j\cot(\pi w)
$$
and the $I$-Bessel function $I_\nu$, certain numbers $\k(a,b,c)$ defined in equation \eqref{E:kappadef} and the Kloostermann sums $K_k(n,\nu)$ given in \eqref{E:KloostermannUnrestricted}.

\begin{theorem}\label{thm:ul}
	For $\ell\in2\N_0$ and $n\to\infty$, we have 
	\begin{multline*}
		u_{\ell}(n)=\frac{\pi}{2^\frac34\sqrt{3}(24n+1)^\frac34} \sum_{\substack{0 \leq j \leq \frac{\ell}{2}\\ a,b,c \geq 0\\a+b+c= j}} \binom{\ell}{2j}  \frac{\k(a,b,c)}{2^{\frac{a}{2}+c}}(24n+1)^{\frac{a}{2}+c} \sum_{\substack{1\le k\le  \sqrt{n}\\0\leq \nu\leq 2k-1}} k^{a-2}K_k(n,\nu) \\
		\times \int_{-1}^1 C_{\ell-2j}\left(\frac{1}{2k}\left(\frac{x}{\sqrt{6}}-\nu-\frac12\right)\right) \frac{I_{a+2c-\frac{3}{2}}\left(\frac{\pi}{3\sqrt{2}k}\sqrt{\left(24n+1\right)(1-x^2)}\right)}{(1-x^2)^{\frac{a}{2}+c-\frac{3}{4}}}dx \\+O\left(n^{\frac{5\ell}{4}-\frac{1}{2}}+n^{-\frac{1}{4}}\log(n)\right).
	\end{multline*}
\end{theorem}

We consider also \textit{Durfee unimodal sequences} introduced by Kim and Lovejoy \cite{KL2}, which are those sequences \eqref{E:unimodaldef} where $k$ is the size of the Durfee square (see \cite{Andrews}) of the partition $\sum_{j=1}^ra_j$ and $b_s\le c-k$. Let $v(m,n)$ be the number of Durfee unimodal sequences of size $n$ and rank $m$.  Standard techniques (see \cite[proof of Proposition 3.1]{KL2}) show that the generating function is 
\[
	V(\z;q) = \sum_{\substack{n\ge0\\m\in\Z}} v(m,n)q^n\z^m = \sum_{0\le\ell\le m} \frac{ (q;q)_mq^{\ell^2+m}\z^\ell}{(\z q,q;q)_\ell(\z^{-1}q,q;q)_{m-\ell}}.
\]
Define the moments $v_\ell(n)$ as above. Although not obvious from the definition or the above generating function, $V(\z;q)$ is symmetric under $\z\mapsto\z^{-1}$.\footnote{This follows from \cite[Proposition 3.1]{KL2},
\[
	V(\z;q) = \sum_{m\ge0} \frac{\left(q^{m+1};q\right)_mq^m}{\left(\z q,\z^{-1}q;q\right)_m}.
\]} Again note that $v_\ell(n)=0$ for $\ell$ odd, the main asymptotic term and resulting logistic distribution may be found in \cite{BJSM}. Asymptotics for $v_\ell(n)$ are significantly more challenging than for $u_\ell(n)$ because $V(\z;q)$ depends on a shifted version of $\psi$ (see Lemma \ref{L:genfnsrewrite}).  Unlike for unrestricted unimodal sequences, the shifted $\psi$ has terms that contribute to asymptotic growth (compare Corollaries \ref{C:generalpsitranshknotation} and \ref{C:psitransdurfeecase}).  This is reflected by the terms below in the second sum on $k$.

To state the asymptotic series for $\ell$ even, we require some notation. Let $K:=\frac{k}{\gcd(k,6)}$ and define $\g_{k,\vp,\nu}$, $\d_{k,\vp,\nu}$, and $\e_{k,\vp}^\pm$ as in \eqref{E:gammadef}, \eqref{E:deltadef}, and \eqref{E:epsilonpmdef}, respectively. Let $\calK_{k,\vp,\nu}(n)$, $\calK_{k,\vp,j_2,j_3,a}^{\pm,[1]}(n)$, and $\calK_{k,\vp,j_3,a}^{\pm,[2]}(n)$ be the Kloostermann-type sums defined in \eqref{E:DurfeeKloostermann}, \eqref{E:erfDurfeeKloostermann1}, and \eqref{E:erfDurfeeKloostermann2}, respectively. We use the notation $\sum_\pm$ to denote a sum of two terms depending on the sign.

\begin{theorem}\label{thm:u2}
	For $\ell\in2\N_0$ and $n\to\infty$, we have
	\begin{align*}
		&v_\ell(n) = \frac{\pi}{2^\frac123^\frac94(4n+1)^\frac34}\sum_{\substack{0\le j\le\frac\ell2\\a,b,c\ge0\\a+b+c=j}} \binom{\ell}{2j}\frac{3^{\frac a2+c}\k(a,b,c)}{2^{\ell-2j}}(4n+1)^{\frac a2+c}\hspace{-.7cm}\sum_{\substack{0\le k\le\sqrt{n}\\0\le\nu\le6K-1\\0\le\vp\le\gcd(k,6)-1\\\gcd(\vp,\gcd(k,6))=1}} \hspace{-.5cm} \sqrt{\gcd(k,6)}k^{a-1}\calK_{k,\vp,\nu}(n)\\
		&\times \int_{-1}^1 \left(C_{\ell-2j}\left(\tfrac{1}{6k}\left(x-\gcd(k,6)\left(\nu+\tfrac12\right)+\tfrac k2-\vp\right)\right) - \d_{k,\vp,\nu}(-1)^{\frac\ell2+j}2(\ell-2j)! \left(\tfrac{3k}{\pi(x+\g_{k,\vp,\nu})}\right)^{\ell-2j+1}\right)\\
		&\hspace{9cm}\times  \frac{I_{a+2c-\frac32}\left(\frac{\pi}{\sqrt{3}k}\sqrt{(4n+1)\left(1-x^2\right)}\right)} {\left(1-x^2\right)^{\frac a2+c-\frac34}} dx\\
		&+\sum_{\substack{1\le k\le\sqrt{n}\\0\le\vp\le\gcd(k,6)-1\\\gcd(\vp,\gcd(k,6))=1}} \sum_\pm \e_{k,\vp}^\pm\frac{\sqrt{\gcd(k,6)}}{4n+1} \left(\sum_{\substack{j_1,j_3\ge0,j_2\ge1\\2j_1+j_2+j_3=\ell}} \sum_{\substack{a,b,c\ge0\\a+b+c=j_1}} \sum_{0\le j_4\le\frac{j_2-1}{2}}\right.\\
		&\hspace{2cm}\frac{\ell!3^{\frac a2+c+\frac{j_2}{2}+\frac{j_3}{2}+\frac{j_4}{2}-1}k^{a-\frac12+j_4}}{(2j_1)!j_2\cdot j_3!j_4!(j_2-1-2j_4)!2^{\ell+a-2j_1+j_4-1}\gcd(k,6)^{j_3}\pi^{j_4}}\k(a,b,c) \calK_{k,\vp,j_2,j_3,a}^{\pm,[1]}(n)\\
		&\hspace{5.5cm}\times (4n+1)^{\frac a2+c+\frac{j_2}{2}+\frac{j_3}{2}-\frac{j_4}{2}} I_{a+2c+j_2+j_3-j_4-2}\left(\frac{\pi}{\sqrt{3}k}\sqrt{4n+1}\right)\\
		&\pm \sum_{\substack{j_1,j_3\ge0\\2j_1+j_3=\ell}} \sum_{\substack{a,b,c\ge0\\a+b+c=j_1}} \binom{\ell}{2j_1}\k(a,b,c)\frac{3^{\frac a2+c+\frac{j_3}{2}-\frac{3}{2}}k^{a-\frac32}\pi}{2^{a+j_3-1}\gcd(k,6)^{j_3}} \calK_{k,\vp,j_3,a}^{\pm,[2]}(n)(4n+1)^{\frac a2+c+\frac{j_3}{2}-\frac{1}{2}}\\
		&\hspace{3.5cm}\left.{\vphantom{\sum_{\substack{j_1,j_3\ge0,j_2\ge1\\2j_1+j_2+j_3=\ell}}}}\times \int_0^1 t^{-a-2c-j_3+1}I_{a+2c+j_3-1}\left(\frac{\pi t}{\sqrt{3}k}\sqrt{4n+1}\right) dt\right) + O\left(n^{\frac{5\ell}{4}-\frac12}+\frac{\log(n)}{n^\frac14}\right).
	\end{align*}
\end{theorem}

\begin{remark}
	If $\ell=0$, then Theorem \ref{thm:u2} becomes 
	\begin{multline*}
		v_0(n) = \frac{\pi}{2^\frac123^\frac94(4n+1)^\frac34}\sum_{\substack{0\le k\le\sqrt{n}\\0\le\nu\le6K-1\\0\le\vp\le\gcd(k,6)-1\\\gcd(\vp,\gcd(k,6))=1}} \hspace{-.5cm} \frac{\sqrt{\gcd(k,6)}\calK_{k,\vp,\nu}(n)}{k}\int_{-1}^1 I_{-\frac32}\left(\frac{\pi}{\sqrt{3}k}\sqrt{(4n+1)\left(1-x^2\right)}\right)\\
		\times \left(1-x^2\right)^\frac34 \left(\cot\left(\tfrac{\pi}{6k}\left(x-\gcd(k,6)\left(\nu+\tfrac12\right)+\tfrac k2-\vp\right)\right) - \d_{k,\vp,\nu}\tfrac{6k}{\pi(x+\g_{k,\vp,\nu})}\right) dx\\
		+\sum_{\substack{1\le k\le \sqrt{n}\\0\le\vp\le\gcd(k,6)-1\\\gcd(\vp,\gcd(k,6))=1}} \sum_\pm \e_{k,\vp}^\pm\frac{\sqrt{\gcd(k,6)}}{4n+1} \left(\frac{2\calK_{k,\vp,0,0,0}^{\pm,[1]}(n)}{3\sqrt{k}} I_{-2}\left(\frac{\pi}{\sqrt{3}k}\sqrt{4n+1}\right)\right.\\
		\left.\pm \frac{2\pi \calK_{k,\vp,0,0}^{\pm,[2]}(n)}{(3k)^\frac32\sqrt{4n+1}}\int_0^1 tI_{-1}\left(\frac{\pi t}{\sqrt{3}k}\sqrt{4n+1}\right) dt\right) + O\left(\frac{\log(n)}{n^\frac14}\right).
	\end{multline*}
	Following the method of \cite{CN}, it should be possible to improve the above to an exact formula.
\end{remark}

The paper is organized as follows. In Section \ref{sec:pre} we recall basic facts about some special functions as well as transformation properties of the $\eta$-function, the $\vth$-function, and the crank-generating function. In Section \ref{sec:modtrans} we prove Theorem \ref{T:psitrans} and show further transformation properties. In Section \ref{sec:mordellint}, we split off contributions of the occuring Eichler integrals. In Section \ref{sec:taylorcoeff} we determine asymptotic expansions of various functions and in Section \ref{sec:circmeth} we finish the proofs of Theorem \ref{thm:ul} and Theorem \ref{thm:u2} using the Circle Method.

\section*{Acknowledgments}

We thank Caner Nazaroglu for useful conversations, and we thank William Craig and Joshua Males for their helpful comments on an earlier draft.  We thank the two anonymous referees for providing useful feedback which improved the exposition.

\section{Preliminaries}\label{sec:pre}

\subsection{Special functions}

Following \cite[pp. 404--405]{Lehner} and in particular using the representation of the $I$-Bessel function as a loop integral (as given in \cite{Watson}, \S 6.22, equation (1)) we conclude the following lemma.

\begin{lemma}\label{L:BesselBound}
	Suppose that $k\in\N$, $\nu\in\R$ and $\vth_1,\vth_2,A,B\in\R^+$ satisfy $k\ll\sqrt{n}$, $A\ll\frac nk$, $B\ll\frac1k$, and $k\vth_1$, $k\vth_2\asymp\frac{1}{\sqrt{n}}$. Then we have
	\[
		\int_{\frac kn-ik\vth_1}^{\frac kn+ik\vth_2} Z^{-\nu}\exp\left(AZ+\frac BZ\right)dZ = 2\pi i\left(\frac AB\right)^\frac{\nu-1}{2}I_{\nu-1}\left(2\sqrt{AB}\right) +
		\begin{cases}
			O\left(n^{\nu-\frac12}\right) & \text{if }\nu\ge0,\\
			\vspace{-.45cm}\\
			O\left(n^\frac{\nu-1}{2}\right) & \text{if }\nu<0.
		\end{cases}
	\]
\end{lemma}

 Our proof is based on \cite[pp. 404--405]{Lehner} and is reproduced here for completeness.

\begin{proof}[Proof sketch.]\phantom{\qedhere} 
	\hspace{-1cm} The following integral representation follows from \cite{Watson}, \S 6.22, equation (1):
	
	\begin{multline*}
		\hspace{-2mm} 2\pi i\left(\frac AB\right)^\frac{\nu-1}{2}I_{\nu-1}\left(2\sqrt{AB}\right) = \left(\int_{-\infty}^{-\frac kn}+\int_{-\frac kn}^{-\frac kn-ik\vth_1}+\int_{-\frac kn-ik\vth_1}^{\frac kn-ik\vth_1}+ \int_{\frac kn-ik\vth_1}^{\frac kn+ik\vth_2}+\int_{\frac kn+ik\vth_2}^{-\frac kn+ik\vth_2}+ \int_{-\frac kn+ik\vth_2}^{-\frac kn} \right. \\ \left. +\int_{-\frac kn}^{-\infty}\right) Z^{-\nu}\exp\left(AZ+\frac BZ\right) dZ =: J_1 + J_2 + \ldots + J_7.
	\end{multline*}
	Note that $J_4$ is the left-hand side of Lemma \ref{L:BesselBound}, so it remains to bound the other integrals.
	
	For $J_2$ and $J_6$, we use $\re(Z)=-\frac kn<0$ and $\re(\frac1Z)=\frac{\re(Z)}{|Z|^2}<0$, to obtain $|J_2|,|J_6|\ll n^{\nu-\frac{1}{2}}$. For $J_3$ and $J_5$, we use $\re(Z)\leq \frac kn<0$ and $\re(\frac1Z)\ll k$, to show that
	$$
	|J_3|,|J_5|\ll \begin{cases} n^{\nu-\frac{1}{2}} & \text{if $\nu\geq 0,$} \\ n^{\frac{\nu-1}{2}} & \text{if $\nu < 0$.} \end{cases}
	$$
	Finally, we combine
	\begin{align*}
	J_1 + J_7  &=-2i\sin(\pi \nu)\int_{-\frac{k}{n}}^{-\infty} |Z|^{-\nu}\exp\left(AZ+\frac{B}{Z}\right) dZ.
	\end{align*}
	The lemma follows by bounding
	\begin{numcases}{|J_1+J_7| \ll}
		n^{\nu-\frac{1}{2}} & \text{if $\nu>1,$} \notag \\ 
		n^{\frac{\nu-1}{2}} & \text{if $\nu \leq 1.$} \tag*{$\square$}
	\end{numcases}  
\end{proof}

We also require certain functions occurring in the transformation law of our unimodal rank generating functions. Define for $z\in\C$, $\t\in\H$
\[
	f_\nu(z;\t) := \frac{\sin(\pi z)}{\sinh\left(\frac{\pi z}{\t}\right)} e^\frac{\pi\nu z^2}{\t}.
\]
By Lemma 3.1 of \cite{BMR}, we have
\begin{equation}\label{E:f_kdef}
	f_\nu(z;\t) = \sum_{r\ge0} \frac{(2\pi iz)^{2r}}{(2r)!}\sum_{\substack{a,b,c\geq0\\a+b+c=r}} \nu^a\k(a,b,c)\t^{1-a-2c},
\end{equation}
where, with $B_\nu(x)$ the $\nu$-th Bernoulli polynomial,
\begin{equation}\label{E:kappadef}
	\k(a,b,c) := \frac{(2(a+b+c))!(-1)^{a+c}B_{2c}\left(\frac12\right)}{a!(2b+1)!(2c)!\pi^a4^{a+b}}.
\end{equation}
We let $b_r(\nu;\t):=\sum_{\substack{a,b,c\ge0\\a+b+c=r}}\nu^a\k(a,b,c)\t^{1-a-2c}$. A direct calculation yields the following lemma.

\begin{lemma}\label{lem:boundsfnuck}
	Let $k\ll\sqrt{n}$ and suppose that $Z=\frac kn-ik\TH$ with $-\vth_1\le\TH\le\vth_2$, where $k\vartheta_1, k\vartheta_2 \asymp \frac{1}{\sqrt{n}}$. Then we have 
	\[
		b_r(jk;Z) \ll_j |Z|^{1-2r}.
	\]
\end{lemma}

Finally, we require the following integral evaluation (see below equation (3,4) of \cite{CN}).

\begin{lemma}\label{lem:lem24}
	 Let $V\in\C$ with $\re(V)>0$ and $a\in\R\setminus\{0\}$. Then for any $\e>0$,
	\[
		\int_{-V}^{\infty-i\e} \frac{e^{-\pi a^2\zz}}{\sqrt{-\zz}} d\zz = \frac{1}{ia}\left(\sgn(a)+\erf\left(ia\sqrt{\pi V}\right)\right),
	\]
	where the path of integration avoids the branch cut of $\sqrt{-\zz}$ along the positive real axis.
\end{lemma}

\subsection{Modular forms and Jacobi forms}

We define $\calC^*(z;\t):=q^{-\frac{1}{24}}{C}(\z;q)$. The classical {\it Jacobi theta function} is given as
$$
	\vartheta(z;\tau):=i\sum_{m \in \mathbb{Z}+\frac{1}{2}}(-1)^{m-\frac{1}{2}}q^{\frac{m^2}{2}}\zeta^m.
$$
For $\pmat{a&b\\c&d}\in\SL_2(\Z)$, the multiplier for the Dedekind eta function is defined as 
\begin{equation}\label{E:etamultiplier}
	\chi\abcd :=
	\begin{cases}
		\left(\frac{d}{|c|}\right)e^{\frac{\pi i}{12}\left((a+d)c-bd\left(c^2-1\right)-3c\right)} & \text{if }c\text{ is odd},\\
		\vspace{-.4cm}\\
		\left(\frac cd\right)e^{\frac{\pi i}{12}\left(ac\left(1-d^2\right)+d(b-c+3)-3\right)} & \text{if }c\text{ is even},
	\end{cases}
\end{equation}
where $(\frac\cdot\cdot)$ is the Kronecker symbol. We have the following transformation laws (see \cite{Kn}).

\begin{lemma}\label{L:etathetaCstartrans}
	For $\re(Z)>0$ and $\gcd(h,k)=1$, we have 
	\begin{align*}
		\eta\left(\frac hk+\frac{iZ}{k}\right) &= \chi\mat{[-h]_k&-\frac{h[-h]_k+1}{k}\\k&-h}^{-1} \frac{1}{\sqrt{iZ}} \eta\left(\frac{[-h]_k}{k}+\frac{i}{kZ}\right),\\
		\vth\left(z;\frac hk+\frac{iZ}{k}\right) &= \chi\mat{[-h]_k&-\frac{h[-h]_k+1}{k}\\k&-h}^{-3} \frac{1}{\sqrt{iZ}} e^{-\frac{\pi kz^2}{Z}} \vth\left(\frac{z}{iZ};\frac{[-h]_k}{k}+\frac{i}{kZ}\right),\\
		\calC^*\left(z;\frac hk+\frac{iZ}{k}\right) &= \frac{\sin(\pi z)}{\sin\left(\frac{\pi z}{iZ}\right)} \chi\mat{[-h]_k&-\frac{h[-h]_k+1}{k}\\k&-h} \frac{1}{\sqrt{iZ}} e^\frac{\pi kz^2}{Z} \calC^*\left(\frac{z}{iZ};\frac{[-h]_k}{k}+\frac{i}{kZ}\right).
	\end{align*}
\end{lemma}

\section{Modular transformations and the proof of Theorem \ref{T:psitrans}}\label{sec:modtrans}

\subsection{Rewriting the generating functions}

We use identities of Ramanujan \cite{Andrews} and Kim--Lovejoy \cite{KL,KL2} to rewrite $U(\z;q)$ and $V(\z;q)$ in terms of $\vth,\eta,\calC^*,\psi$, and certain sparse series. Define
\begin{align*}
	\calU_1(z;\t) &:= -\frac{i\calC^*(z;\t)}{2\eta(\t)}(\vth(2z;\t)+\psi(2z;\t)),\qquad \calV_1(z;\t) := -q^\frac{1}{12}\z^{-\frac12}\frac{\calC^*(z;\t)}{\eta(\t)}\psi\left(3z-\t+\frac12;6\t\right),\\
	H_1(\z;q) &:= (1-\z)\sum_{m\ge0} (-1)^m\left(1-\z^2q^{2m+1}\right)q^\frac{m(3m+1)}{2}\z^{3m},\qquad H_2(\z;q) := (1-\z)\sum_{m\ge0} q^{m(m+1)}\z^m.
\end{align*}

\begin{lemma}\label{L:genfnsrewrite}
	Let $z\in\R$.
	\begin{enumerate}[leftmargin=*,label=\rm{(\arabic*)}]
		\item We have
		\[
			U(\z;q) = q^{-\frac{1}{24}}\calU_1(z;\t) + H_1(\z;q).
		\]

		\item For $\ell\in2\N_0$, we have
		\[
			\left[\frac{\del^\ell}{\del z^\ell} V(\z;q)\right]_{z=0} = \left[\frac{\del^\ell}{\del z^\ell} \left(q^{-\frac14}\calV_1(z;\t)+H_2(\z;q)\right)\right]_{z=0}.
		\]
	\end{enumerate}
\end{lemma}

\begin{proof}
	(1) Since $\sgn(m+\frac{z_2}{\t_2})=\sgn(m)$ for $z\in\R$, the claim follows directly from \cite[entry 6.3.2]{LostNotebook} (see also \cite[Proposition 2.1]{KL}).\\
	(2) We begin with (see \cite[Proposition 3.1]{KL2}), 
	\[
		V(\z;q) = C(\z;q)G_2(\z;q) + H_2(\z;q),
	\]
	where
	\[
		G_2(\z;q) := \frac{1}{(q;q)_\infty}\sum_{m\ge0} \left(1-\z q^{2m+1}\right)   q^{3m^2+2m}\z^{3m+1}.
	\]
	
	We next rewrite $G_2$ as $G_2(\z;q)=G_2^{[1]}(\z;q)+G_2^{[2]}(\z;q)$, where
	\[
	G_2^{[1]}(\z;q) := \frac{1}{(q;q)_{\infty}}\sum_{m\ge0} q^{3m^2+2m}\z^{3m+1},\qquad
	G_2^{[2]}(\z;q) := -\frac{1}{(q;q)_{\infty}}\sum_{m\ge0} q^{3m^2+4m+1}\z^{3m+2}.
	\]
	In $G_2^{[2]}(\z;q)$, we make the change of variables $m\mapsto-m-1$ to obtain
	\[
		G_2^{[2]}(\z;q)=-\frac{1}{(q;q)_{\infty}}\sum_{m\le-1} q^{3m^2+2m} \z^{-3m-1}.
	\]
	Next note that for $\ell$ even
	\[
		\left[\frac{\del^\ell}{\del z^\ell}C(\z;q)G_2(\z;q)\right]_{z=0} = \left[\frac{\del^\ell}{\del z^\ell}C(\z;q) \left(G_2^{[1]}(\z;q)+G_2^{[2]}\left(\z^{-1};q\right)\right)\right]_{z=0}.
	\]
	Now combine, for $z\in\R$,
	\[
		G_2^{[1]}(\z;q) + G_2^{[2]}\left(\z^{-1};q\right) = -\frac{q^{-\frac{1}{4}}\z^{-\frac12} }{(q;q)_{\infty}} \psi\left(3z-\t+\frac12;6\t\right),
	\]
	where we use that for $z\in\R$ we have $|\frac{\Im(3z-\t+\frac12)}{\Im(6\t)}|=\frac16<\frac12$. From this, we obtain the claim.
\end{proof}

\subsection{Proof of Theorem \ref{T:psitrans}}

We are now ready to prove Theorem \ref{T:psitrans}.

\begin{proof}[Proof of Theorem \ref{T:psitrans}]
	 We first note that
	\begin{equation}\label{limit}
		\lim_{t\to\infty} \wh\psi(z;\t,\t+it+\e) = \psi(z;\t),
	\end{equation}
	which follows from
	\[
		\lim_{|z|\to\infty} \erf(z) =
		\begin{cases}
			1 & \text{if }|\Arg(z)|<\frac\pi4,\\
			-1 & \text{if }\frac{3\pi}{4}<|\Arg(z)|\le\pi.
		\end{cases}
	\]
	We write ($w\in \H $ can be choosen freely)
	\[
		\wh\psi(z;\t,w) = \psi(z;\t) + \psi^*(z;\t,w).
	\]
	Using that $\erf'(x)=\frac{2}{\sqrt{\pi}}e^{-x^2}$, we obtain, away from the branch cut,
	$$
		\frac{\del}{\del w}\wh\psi(z;\t,w)=\frac{i}{\sqrt{i(w-\tau)}}\sum_{m \in \mathbb{Z}+\frac{1}{2}}(-1)^{m-\frac{1}{2}}\left(m+\frac{z_2}{\t_2}\right)q^{\frac{m^2}{2}}\zeta^me^{\pi i(w-\t)\left(m+\frac{z_2}{\t_2}\right)^2}.
	$$
	From \eqref{limit} we thus obtain 
	\[
		\psi^*(z;\t,w) = -ie^{-\frac{\pi i\t z_2^2}{\t_2^2}} \int_w^{\t+i\infty+\e} e^\frac{\pi iz_2^2\zz}{\t_2^2} \frac{\sum_{m\in\Z+\frac12} (-1)^{m-\frac12}\left(m+\frac{z_2}{\t_2}\right)e^{\pi i\left(m^2\zz+2m\left(z+(\zz-\t)\frac{z_2}{\t_2}\right)\right)}}{\sqrt{i(\zz-\t)}} d\zz.
	\]
	Now recall that by \Cref{T:CNtrans} we have, for $\pabcd\in\SL_2(\Z)$,
	\begin{equation*}\label{E:psihatmodularity}
		\wh\psi\left(\frac{z}{c\t+d};\frac{a\t+b}{c\t+d},\frac{aw+b}{cw+d}\right) = \e_{\t,w}\abcd \chi\abcd^3 (c\t+d)^\frac12 e^\frac{\pi icz^2}{c\t+d} \wh{\psi}(z;\t,w),
	\end{equation*}
	where 
	\begin{equation}\label{E:epsilonbranchcutfactor}
		\e_{\t,w}\abcd := \sqrt{\frac{i(w-\t)}{(c\t+d)(cw+d)}} \frac{\sqrt{c\t+d}\sqrt{cw+d}}{\sqrt{i(w-\t)}}.
	\end{equation}
	Then we have 
	\begin{multline*}
		\psi\left(\frac{z}{c\t+d};\frac{a\t+b}{c\t+d}\right) - \e_{\t,w}\abcd \chi\abcd^3 (c\t+d)^\frac12 e^\frac{\pi icz^2}{c\t+d} \psi(z;\t)\\
		= -\psi^*\left(\frac{z}{c\t+d};\frac{a\t+b}{c\t+d},\frac{aw+b}{cw+d}\right) + \e_{\t,w}\abcd \chi\abcd^3 (c\t+d)^\frac12 e^\frac{\pi icz^2}{c\t+d} \psi^*(z;\t,w).
	\end{multline*}
	Now letting $w\to\t+i\infty+\e$ the integral defining $\psi^*(z;\t,w)$ vanishes and we are left with
	\begin{multline*}
		-\lim_{w\to\t+i\infty+\e} \psi^*\left(\frac{z}{c\t+d};\frac{a\t+b}{c\t+d},\frac{aw+b}{cw+d}\right) = ie^{-\pi i\frac{a\t+b}{c\t+d}\left(\frac{\im\left(\frac{z}{c\t+d}\right)}{\im\left(\frac{a\t+d}{c\t+d}\right)}\right)^2} \int_\frac ac^{\frac{a\t+b}{c\t+d}+i\infty+\e} e^{\pi i\left(\frac{\im\left(\frac{z}{c\t+d}\right)}{\im\left(\frac{a\t+ d}{c\t+d}\right)}\right)^2\zz}\\
		\times \frac{\sum_{m\in\Z+\frac12} (-1)^{m-\frac12} \left(m+\frac{\im\left(\frac{z}{c\t+d}\right)}{\im\left(\frac{a\t+d}{c\t+d}\right)}\right) e^{\pi i\left(m^2\zz+2m\left(\frac{z}{c\t+d}+\left(\zz-\frac{a\t+b}{c\t+d}\right) \frac{\im\left(\frac{z}{c\t+d}\right)}{\im\left(\frac{a\t+d}{c\t+ d}\right)}\right)\right)}}{\sqrt{i\left(\zz-\frac{a\t+b}{c\t+d}\right)}} d\zz,
	\end{multline*}
	using that $c>0$ and computing $\lim_{w\to\t+i\infty+\e}\e_{\t,w}\pmat{a&b\\c&d}=1$. The claim now follows from
	\[
		\frac{a\t+b}{c\t+d} = \frac ac - \frac{1}{c(c\t+d)}. \qedhere
	\]
\end{proof}

Setting $a:=[-h]_k$, $b:=-\frac{h[-h]_k+1}{k}$, $c:=k$, $d:=-h$, and $\t:=\frac hk+\frac{iZ}{k}$ in Theorem \ref{T:psitrans}, we immediately obtain the following corollary.

\begin{corollary}\label{C:generalpsitranshknotation}
	For $z\in\C$, we have
	\begin{multline*}
		\psi\left(z;\frac hk+\frac{iZ}{k}\right) = \chi\mat{[-h]_k&-\frac{h[-h]_k+1}{k}\\k&-h}^{-3} (iZ)^{-\frac12} e^{-\frac{\pi kz^2}{Z}}\\
		\times \left(\psi\left(\frac{z}{iZ};\frac{[-h]_k}{k}+\frac{i}{kZ}\right) - ie^{\frac{\pi k}{Z}\left(\frac{\re\left(z\ol Z\right)}{\re(Z)}\right)^2} \calE_\frac{[-h]_k}{k}\left(\frac{z}{iZ};\frac{[-h]_k}{k}+\frac{i}{kZ}\right)\right).
	\end{multline*}
\end{corollary}

\subsection{Unrestricted unimodal sequences}

Using Corollary \ref{C:generalpsitranshknotation} and Lemma \ref{L:etathetaCstartrans}, the function $\calU_1$ transforms as follows.

\begin{theorem}\label{thm:u1}
	We have for $Z\in\C$ with $\re(Z)>0$, $0\le h<k,\gcd(h,k)=1$, $z\in\R$ with $|kz|<\frac14$ 
	\begin{multline*}
		\calU_1\left(z;\frac hk+\frac{iZ}{k}\right) = -i \chi\mat{[-h]_k&-\frac{h[-h]_k+1}{k}\\k&-h}^{-1} \left(iZ\right)^{-\frac12}\\
		\times \left({}f_{3k}\left(z;-Z\right) \calU_1\left(\frac{z}{iZ};\frac{[-h]_k}{k}+\frac{i}{kZ}\right) + \tfrac12 f_k(z;Z) \frac{\calC^*\left(\frac{z}{iZ}; \frac{[-h]_k}{k}+\frac{i}{kZ}\right)}{\eta\left(\frac{[-h]_k}{k}+\frac{i}{kZ}\right)} \calE_\frac{[-h]_k}{k}\left(\frac{2z}{iZ};\frac{[-h]_k}{k}+\frac{i}{kZ}\right)\right).
	\end{multline*}	
\end{theorem}

\subsection{Durfee unimodal sequences}

Recall that $K=\frac{k}{\gcd(k,6)}$. We further set $H:=\frac{6h}{\gcd(k,6)}$ and $\a_{H,K}:= \frac K2-\frac H6$. Before stating the relevant transformation, we state the following property of the rational numbers $\a_{H,K}$. A direct calculation gives the following.

\begin{lemma}\label{L:alphabadms}
	We have
	\begin{multline*}
		\left\{m \in \mathbb{Z}+\frac{1}{2}: |m-\alpha_{H,K}|\leq \frac{1}{\gcd(k,6)}\right\}  \\
		= \begin{cases}
			\{\alpha_{H,K}, \alpha_{H,K}\pm 1\} & \text{if $\gcd(k,12)=1$,} \\ \left\{\alpha_{H,K} \pm \frac{1}{2}\right\} & \text{if $\gcd(k,12)=2$,}\\
			\{\alpha_{H,K}\}  & \text{if $\gcd(k,12) =4$,} \\
			\left\{\alpha_{H,K}+\frac{1}{\gcd(k,6)}\right\}  & \text{if $\gcd(k,12) \in \{3,6\}$ and $h \equiv 1\pmod{3}$,}\vspace{1mm} \\
			\left\{\alpha_{H,K}-\frac{1}{\gcd(k,6)}\right\}  & \text{if $\gcd(k,12) \in \{3,6\}$ and $h \equiv 2\pmod{3}$,}  \\
			0 & \text{if $\gcd(k,12)= 12.$}
		\end{cases}
	\end{multline*}
\end{lemma}

We analyze the terms with $|m-\a_{H,K}|\le\frac{1}{\gcd(k,6)}$ in $\psi$ and $\calE$ differently, and to this end we define
\begin{align*}
	\Psi(z;\t) &:= i{{\sum}^*_{m\in\Z+\frac12}} \sgn\left(m+\frac{z_2}{\t_2}\right)(-1)^{m-\frac12}q^\frac{m^2}{2}\z^m,\\[-12pt]
	\E_\frac ac(z;\t) &:= e^{-\frac{\pi iaz_2^2}{c\t_2^2}}\int_\frac ac^{\t+i\infty+\e} e^{\pi i\frac{z_2^2}{\t_2^2}\zz} \frac{{{\sum}^*_{m\in\Z+\frac12}} (-1)^{m-\frac12}\left(m+\frac{z_2}{\t_2}\right) e^{\pi i\left(m^2\zz+2m\left(z+(\zz-\t)\frac{z_2}{\t_2}\right)\right)}}{\sqrt{i(\zz-\t)}} d\zz,
\end{align*}
where ${{\sum}^*_{m\in\Z+\frac12}}$ denotes the summation restricted to those $m\in\Z+\tfrac12$ with $|m-\alpha_{H,K}|>\frac{1}{\gcd(k,6)}$.

To state our transformations, it is convenient to define $\b:=3z-\frac hk+\frac12$. Note that $\b\in\R$ if and only if $z\in\R$. Let $W:=\frac{6}{\gcd(k,6)}Z$ and denote by ${{\sum}^{**}_{m\in\Z+\frac12}}$ the summation restricted to these $m\in\Z+\frac12$ with $|m-\a_{H,K}|\leq\frac{1}{\gcd(k,6)}$.

\begin{corollary}\label{C:psitransdurfeecase}
	For $z\in\R$, we have 
	\begin{align*}
		&\psi\left(3z-\frac hk-\frac{iZ}{k}+\frac12;6\left(\frac hk+\frac{iZ}{k}\right)\right)\\
		&\hspace{.33cm}= \chi\mat{[-H]_K&-\frac{H[-H]_K+1}{K}\\K&-H}^{-3} (iW)^{-\frac12} e^{-\frac{\pi K}{W}\left(\b-\frac{iZ}{k}\right)^2}\\
		&\hspace{.66cm}\times \left( \Psi\left(\frac{\b}{iW}-\frac{1}{6K};\frac{[-H]_K}{K}+\frac{i}{KW}\right) - ie^{\frac{\pi K\b^2}{W}} \E_\frac{[-H]_K}{K}\left(\frac{\b}{iW}-\frac{1}{6K}; \frac{[-H]_K}{K}+\frac{i}{KW}\right)\right.\\
		&\hspace{1cm}\left. + i{{\sum}^{**}_{m\in\Z+\frac12}} (-1)^{m-\frac12} e^{\pi im^2\left(\frac{[-H]_K}{K}+\frac{i}{KW}\right) + 2\pi im\left(\frac{\b}{iW}-\frac{1}{6K}\right)} \erf\left(i(3Kz+\alpha_{H,K}-m)\sqrt{\frac{\pi }{KW}}\right)\right).
	\end{align*}
\end{corollary}

\begin{proof}
	Corollary \ref{C:generalpsitranshknotation} implies that
	\begin{multline*}
		\psi\left(3z-\frac hk-\frac{iZ}{k}+\frac12;6\left(\frac hk+\frac{iZ}{k}\right)\right) = \chi\mat{[-H]_K&-\frac{H[-H]_K+1}{K}\\K&-H}^{-3}(iW)^{-\frac12}e^{-\frac{\pi K}{W}\left(\b-\frac{iZ}{k}\right)^2}\\
		\times \left(\psi\left(\frac{\b}{iW}-\frac{1}{6K};\frac{[-H]_K}{K}+\frac{i}{KW}\right) - ie^\frac{\pi K \b^2}{W}\calE_\frac{[-H]_K}{K}\left(\frac{\b}{iW}-\frac{1}{6K}; \frac{[-H]_K}{K}+\frac{i}{KW}\right)\right).
	\end{multline*}
	 Letting $\zz\mapsto i\zz+\frac{[-H]_K}{K}+\frac{i}{KW}$ in the integral in $\calE$ and simplifying exponents, the terms with $|m-\a_{H,K}|\le\frac{1}{\gcd(k,6)}$ (without the prefactors) become 
	\begin{multline*}
		(-1)^{m+\frac12}e^{\pi im^2\left(\frac{[-H]_K}{K}+\frac{i}{KW}\right)+2\pi im\left(\frac{\b}{iW}-\frac{1}{6K}\right)}\\
		\left(i\sgn(3Kz+\a_{H,K}-m)+(3Kz+\a_{H,K}-m)\int_{-\frac{1}{KW}}^{\infty-i\e} \frac{e^{-\pi(3Kz+\a_{H,K}-m)^2\zz}}{\sqrt{-\zz}} d\zz\right).
	\end{multline*}
	Using Lemma \ref{lem:lem24} we obtain the claim. \qedhere
\end{proof}

We state the transformation of $\calV_1$ we require the following multiplier
\begin{equation*}
	\chi_{h,k} := \frac{\chi\mat{[-h]_k&-\frac{h[-h]_k+1}{k}\\k&-h}^2}{\chi \mat{[-H]_K&-\frac{H[-H]_K+1}{K}\\K&-H}^3}.
\end{equation*}
The following lemma follows from Lemma \ref{L:etathetaCstartrans} and Corollary \ref{C:psitransdurfeecase}.

\begin{lemma}\label{lem:calv}
	With notation as above and $z\in\R$ sufficiently small, we have
	\begin{align*}
		&\calV_1\left(z;\frac hk+\frac{iZ}{k}\right) = i\chi_{h,k}(iW)^{-\frac12}\frac{\calC^*\left(\frac{z}{iZ};\frac{[-h]_k}{k}+ \frac{i}{kZ}\right)}{\eta\left(\frac{[-h]_k}{k}+\frac{i}{kZ}\right)}\\
		&\hspace{1cm}\times \left({\vphantom{e^{\left(\frac{\a_{H,K}^2}{36K}\right)}}}e^{-\frac{\pi ih}{6k}-\frac{6\pi\a_{H,K}z}{W}-\frac{\pi\a_{H,K}^2}{KW}+\frac{\pi i}{6}} f_\frac k2(z;-Z) \Psi\left(\frac{\b}{iW}-\frac{1}{6K};\frac{[-H]_K}{K}+\frac{i}{KW}\right)\right.\\
		&\hspace{1.75cm}+ie^{-\frac{\pi ih}{6k}+\frac{\pi i}{6}} f_k(z;Z)\mathbb{E}_\frac{[-H]_K}{K}\left(\frac{\b}{iW}-\frac{1}{6K};\frac{[-H]_K}{K}+\frac{i}{KW}\right)\\
		&\hspace{2.5cm}\left.+i{{\sum}^{**}_{m\in\Z+\frac12}}(-1)^{m-\frac12}f_{\frac{k}{2}}(z;-Z) \erf\left(i(3Kz+\alpha_{H,K}-m)\sqrt{\frac{\pi }{KW}}\right) \right. \\ & \hspace{4.5cm}  \left. {\vphantom{e^{\left(\frac{\a_{H,K}^2}{36K}\right)}}} \times e^{\frac{\pi z}{Z}(m-\alpha_{H,K})-\frac{\pi}{KW}(m-\alpha_{H,K})^2+\pi i\left(\frac{[-H]_Km^2}{K}+\frac{H}{36K}-\frac{1}{3K}(m-\alpha_{H,K})\right)}\right).
	\end{align*}
\end{lemma}

\section{Mordell-type integrals}\label{sec:mordellint}

\subsection{Unrestricted unimodal sequences}

Similarly as in the proof of Lemma 3.2 of \cite{CN}, we obtain the following representations of $\calE_\frac{[-h]_k}{k}$.

\begin{lemma}\label{L:EichlerRewrite}
 For $z\in\R$ sufficiently small and $\re(Z)>0$, we have
	\[
	\calE_{\frac{[-h]_k}{k}}\left(\frac{2z}{iZ};\frac{[-h]_k}{k}+\frac{i}{kZ}\right) = \frac{1}{\pi i} \sum_{m\in\Z+\frac12} (-1)^{m-\frac12} e^{\pi im^2\frac{[-h]_k}{k}} \lim_{\e\to0^+} \int_{-\infty}^\infty \frac{e^{-\frac{\pi x^2}{kZ}}}{x-(m-2kz)(1+i\e)} dx.
	\]
\end{lemma}

For $0\le D\le\frac{1}{12}$, we write
\begin{equation*}
e^{\frac{2\pi D}{kZ}} \calE_{\frac{[-h]_k}{k}}^*\left(\frac{2z}{iZ};\frac{[-h]_k}{k}+\frac{i}{kZ}\right) = \calE_{\frac{[-h]_k}{k},D}^*\left(\frac{2z}{iZ};\frac{[-h]_k}{k}+\frac{i}{kZ}\right) + \calE_{\frac{[-h]_k}{k},D}^e\left(\frac{2z}{iZ};\frac{[-h]_k}{k}+\frac{i}{kZ}\right),
\end{equation*}
where
\begin{align*} 
	\calE_{\tfrac{[-h]_k}{k},D}^*\left(\tfrac{2z}{iZ};\tfrac{[-h]_k}{k}+\tfrac{i}{kZ}\right) &:= \frac{e^\frac{2\pi D}{kZ}}{\pi i} \sum_{m\in\Z+\frac12} (-1)^{m-\frac12}e^{\pi im^2\frac{[-h]_k}{k}} \lim_{\e\to0^+} \int_{-\sqrt{2D}}^{\sqrt{2D}} \frac{e^{-\frac{\pi x^2}{kZ}}}{x-(m-2kz)(1+i\e)} dx,\\
	\calE_{\tfrac{[-h]_k}{k},D}^e\left(\tfrac{2z}{iZ};\tfrac{[-h]_k}{k}+\tfrac{i}{kZ}\right) &:= \frac{e^\frac{2\pi D}{kZ}}{\pi i}\hspace{-0.2cm}  \sum_{m\in\Z+\frac12} (-1)^{m-\frac12}e^{\pi im^2\frac{[-h]_k}{k}} \lim_{\e\to0^+} \int_{|x|\ge\sqrt{2D}} \frac{e^{\frac{\pi x^2}{kZ}}}{x-(m-2kz)(1+i\e)} dx.
\end{align*}

\begin{lemma}\label{lem:calEbound}
	Assume that $0\le D\le\frac{1}{12}$ and $\re(\frac1Z)\ge\frac k2$. Then, for $\ell{\in\N_0}$
	\[
	\left[\frac{\del^\ell}{\del z^\ell} \calE_{\frac{[-h]_k}{k},D}^e\left(\frac{2z}{iZ};\frac{[-h]_k}{k}+\frac{i}{kZ}\right)\right]_{z=0} \ll \log(k)+k^{\ell}.
	\]
\end{lemma}

\begin{proof}  
	We follow the proof of Lemma 3.3 in \cite{CN}. Assume that $z\in\R$ is sufficiently small. Combining the integral over the negative and positive reals gives, making the change of variables $u=x^2-2D$ 
	\begin{multline*}
		\calE_{\frac{[-h]_k}{k},D}^e\left(\frac{2z}{iZ};\frac{[-h]_k}{k}+\frac{i}{kZ}\right)\\
		= \frac{i}{\pi}\sum_{m\in\Z+\frac12} (-1)^{m-\frac12}(m-2kz)e^{\pi im^2\frac{[-h]_k}{k}}\lim_{\e\to0^+} \int_0^\infty \frac{e^{-\frac{\pi u}{kZ}}}{\sqrt{u+2D}\left(u+2D-(m-2kz)^2(1+i\e)^2\right)} du.
	\end{multline*}
	Using the identity $\frac{1}{a-b}=\frac{a}{b(a-b)}-\frac1b$, we split
	\begin{multline*}
		\frac{1}{u+2D-(m-2kz)^2(1+i\e)^2}\\
		= \frac{u+2D}{(m-2kz)^2(1+i\e)^2(u+2D-(m-2kz)^2(1+i\e)^2)} - \frac{1}{(m-2kz)^2(1+i\e)^2}
	\end{multline*}
	and consider the contribution from each term separately which we denote by $\calE_{\frac{[-h]_k}{k},D,1}^e$ and $\calE_{\frac{[-h]_k}{k},D,2}^e$, respectively. We start with $\calE_{\frac{[-h]_k}{k},D,1}^e$ and write  
	\begin{multline*}
		\calE_{\frac{[-h]_k}{k},D,1}^e\left(\frac{2z}{iZ};\frac{[-h]_k}{k}+\frac{i}{kZ}\right)\\
		= \frac i\pi\sum_{m\in\Z+\frac12} \frac{(-1)^{m-\frac12}e^{\pi im^2\frac{[-h]_k}{k}}}{m-2kz}\lim_{\e\to0^+} \int_0^\infty \frac{\sqrt{u+2D}e^{-\frac{\pi u}{kZ}}}{u+2D-(m-2kz)^2(1+i\e)^2} du.
	\end{multline*}
	The poles of the integrand lie at $u=-2D+(m-2kz)^2(1+i\e)^2$. These have positive real part if $\e$ is sufficiently small. If $\re(\frac1Z)\ge\frac k2$, then either $\re(\frac{e^\frac{\pi i}{4}}{kZ})\ge\frac{1}{2\sqrt2}$ or $\re(\frac{e^{-\frac{\pi i}{4}}}{kZ})\ge\frac{1}{2\sqrt2}$. Using the Residue Theorem, we shift the path of integration to $e^\frac{\pi i}{4}\R^+$ if $\re(\frac{e^\frac{\pi i}{4}}{kZ})\ge\frac{1}{2\sqrt2}$ or to $e^{-\frac{\pi i}{4}}\R^+$ if $\re(\frac{e^{-\frac{\pi i}{4}}}{kZ})\ge\frac{1}{2\sqrt2}$. In the first case we pick up residues which contribute  
	\[
		-2\sum_{m\in\Z+\frac12} \sgn(m)(-1)^{m-\frac12}e^{\pi im^2\frac{[-h]_k}{k}-\frac{\pi\left(m^2-2D\right)}{kZ}} \sum_{j\ge0} \frac{\left(\frac{4\pi z}{Z}\right)^j(m-kz)^j}{j!}.
	\]
	The contribution of the $\ell$-th coefficient in $z$ is $O(k^\ell)$.
	
	We next bound the remaining part and distinguish whether $\re(\frac{e^\frac{\pi i}{4}}{kZ})\ge\frac{1}{2\sqrt2}$ or $\re(\frac{e^{-\frac{\pi i}{4}}}{kZ})\ge\frac{1}{2\sqrt2}$. We first assume that $\re(\frac{e^\frac{\pi i}{4}}{kZ})\ge\frac{1}{2\sqrt2}$ and shift the path of integration to $e^\frac{\pi i}{4}\R^+$. Now the poles are away from the path of integration, and Lebesgue's Theorem on Dominated Convergence allows us to set $\e=0$. Thus, the limit (making the change of variables $u\mapsto e^\frac{\pi i}{4}u$) equals
	\[
		\frac{ie^\frac{\pi i}{4}}{\pi}\sum_{m\in\Z+\frac 12} \frac{(-1)^{m-\frac12}e^{\pi im^2\frac{[-h]_k}{k}}}{m-2kz} \int_0^\infty \frac{\sqrt{\frac{1+i}{\sqrt{2}}u+2D} e^{-\frac{\pi(1+i)u}{\sqrt{2}kZ}}}{\frac{1+i}{\sqrt{2}}u+2D-(m-2kz)^2} du.
	\]
	We bound
	\begin{equation*}
		\hspace{-2mm}\left[\frac{\del^\ell}{\del z^\ell} \frac{1}{(m-2kz)\left(\frac{1+i}{\sqrt{2}}u+2D-(m-2kz)^2\right)}\right]_{z=0}  \ll_\ell \frac{k^{\ell}}{m^3}.
	\end{equation*}
	Now estimating $|e^{-\frac{\pi(1+i)u}{\sqrt{2}}}|\le e^{-\frac{\pi u}{\sqrt{2}}}$ we bound the $\ell$-th derivative of $\calE_{\varrho,D,1}^e$ at $z=0$ against
	\[
		\ll k^\ell \sum_{m\in\Z+\frac12} \frac{1}{m^3} \int_0^\infty \left|\frac{1+i}{\sqrt{2}}u+2D\right|^\frac12 e^{-\frac{\pi u}{\sqrt{2}}} du \ll k^\ell.
	\]
	
	The case $\re(\frac{e^{-\frac{\pi i}{4}}}{kZ})\ge\frac{1}{2\sqrt{2}}$ is treated in the same way.
	
	\indent\indent In $\calE_{\frac{[-h]_k}{k},D,2}^e$, we may take $\e\to0^+$, to obtain 
	\[
		\calE_{\frac{[-h]_k}{k},D,2}^e\left(\frac{2z}{iZ};\frac{[-h]_k}{k}+\frac{i}{kZ}\right) = \frac{1}{\pi i}\sum_{m\in\Z+\frac12} \frac{(-1)^{m-\frac12}e^{\pi im^2\frac{[-h]_k}{k}}}{m-2kz}\int_0^\infty \frac{e^{-\frac{\pi u}{kZ}}}{\sqrt{u+2D}} du.
	\]
	Taylor expanding in $z$, the coefficient of $z^{\ell}$ is
	\[
		\frac{(2k)^\ell}{\pi i}\sum_{m\in\Z+\frac12} \frac{(-1)^{m-\frac12}e^{\pi im^2\frac{[-h]_k}{k}}}{m^{\ell+1}} \int_0^\infty \frac{e^{-\frac{\pi u}{kZ}}}{\sqrt{u+2D}} du.
	\]
	Since $\re(\frac{1}{kZ})\ge\frac12$, the integral is $O(1)$. For $\ell\ge2$, the series converges absolutely, and the overall term is $\ll k^\ell$. For $\ell=0$, the sum is $O(\log(k))$ as in \cite{CN}.
\end{proof}

Proceeding as in (3.11) of \cite{CN}, we obtain the following representation as a Mordell-type integral.

\begin{lemma}\label{lem:calE*}  
	For $z\in\R$ and $\re(Z)>0$, we have  
	\begin{multline*}
		\calE_{\frac{[-h]_k}{k},\frac{1}{12}}^*\left(\frac{2z}{iZ};\frac{[-h]_k}{k}+\frac{i}{kZ}\right)\\
		= \frac{1}{2\sqrt{6}ki}\sum_{\nu=0}^{2k-1} (-1)^\nu e^{\pi i\left(\nu+\frac12\right)^2\frac{[-h]_k}{k}}\int_{-1}^1 \cot\left(\frac{\pi}{2k}\left(\frac{x}{\sqrt{6}}-\nu-\frac12+2kz\right)\right) e^{\frac{\pi}{6kZ}\left(1-x^2\right)} dx.
	\end{multline*}
\end{lemma}

\subsection{Durfee unimodal sequences}

We next rewrite $\E_\frac{[-H]_K}{K}$ as in Lemma \ref{L:EichlerRewrite}.

\begin{lemma}\label{lem:seclem}
	For $z\in\R$ with $|z|$ sufficiently small and $\re(Z)>0$, we have
	\begin{multline*}
		\E_\frac{[-H]_K}{K}\left(\frac{\b}{iW}-\frac{1}{6K};\frac{[-H]_K}{K}+\frac{i}{KW}\right)\\
		= \frac{1}{\pi i}{{\sum}^*_{m\in\Z+\frac12}} (-1)^{m-\frac12}e^{\pi i\left(-\frac{m}{3K}+\frac{[-H]_K}{K}m^2\right)}\lim_{\e\to0^+} \int_{-\infty}^\infty \frac{e^{-\frac{\pi x^2}{KW}}}{x-(m-K\b)(1+i\e)} dx.
	\end{multline*}
\end{lemma}

Let  
\begin{multline}\label{E:wtcalEDdefs}
	\frac{e^\frac{2\pi D}{KW}}{\pi i} {{\sum}^*_{m\in\Z+\frac12}} (-1)^{m-\frac12}e^{\pi i\left(-\frac{m}{3K}+\frac{[-H]_K}{K}m^2\right)} \lim_{\e\to0^+} \int_{-\infty}^\infty \frac{e^{-\frac{\pi x^2}{KW}}}{x-(m-K\b)(1+i\e)} dx\\
	= \E_{\frac{[-H]_K}{K},D}^*\left(\frac{\b}{iW}-\frac{1}{6K};\frac{[-H]_K}{K}+\frac{i}{KW}\right) + \E_{\frac{[-H]_K}{K},D}^e\left(\frac{\b}{iW}-\frac{1}{6K};\frac{[-H]_K}{K}+\frac{i}{KW}\right),
\end{multline}
where 
\begin{align*}
	&\E_{\frac{[-H]_K}{K},D}^*\left(\frac{\b}{iW}-\frac{1}{6K};\frac{[-H]_K}{K}+\frac{i}{KW}\right)\\
	&\hspace{2cm}:= \frac{e^\frac{2\pi D}{KW}}{\pi i}{{\sum}^*_{m\in\Z+\frac12}} (-1)^{m-\frac12} e^{\pi i\left(-\frac{m}{3K}+\frac{[-H]_K}{K}m^2\right)} \lim_{\e\to0^+} \int_{-\sqrt{2D}}^{\sqrt{2D}} \frac{e^{-\frac{\pi x^2}{KW}}}{x-(m-K\b)(1+i\e)} dx,\\
	&\E_{\frac{[-H]_K}{K},D}^e\left(\frac{\b}{iW}-\frac{1}{6K};\frac{[-H]_K}{K}+\frac{i}{KW}\right)\\
	&\hspace{2cm}:= \frac{e^\frac{2\pi D}{KW}}{\pi i}{{\sum}^*_{m\in\Z+\frac12}} (-1)^{m-\frac12} e^{\pi i\left(-\frac{m}{3K}+\frac{[-H]_K}{K}m^2\right)} \lim_{\e\to0^+} \int_{|x|\ge\sqrt{2D}} \frac{e^{-\frac{\pi x^2}{KW}}}{x-(m-K\b)(1+i\e)} dx.
\end{align*}

The following bounds for the Taylor coefficients of $\E_{\frac{[-H]_K}{K},D}^e$ are proved similarly to Lemma \ref{lem:calEbound}.

\begin{lemma}\label{lem:del}
	Let $\ell\in2\N_0$.  With $0\le D\le\frac{1}{2\gcd(k,6)^2}$, we have
	\[
		\left[\frac{\del^\ell}{\del z^\ell} \E_{\frac{[-H]_K}{K},D}^e\left(\frac{\b}{iW}-\frac{1}{6K}; \frac{[-H]_K}{K}+\frac{i}{KW}\right)\right]_{z=0} \ll \log(k) + k^\ell.
	\]
\end{lemma}

We next write $\E_{\frac{[-H]_K}{K},\frac{1}{\gcd(k,6)^2}}^*$ as a Mordell-type integral. For $0\le\nu\le6K-1$ and $0\le\vp<\gcd(k,6)$, we set
\begin{align}\label{E:gammadef}
	\gamma_{k,\vp,\nu}&:=
	\begin{cases}
		1 & \text{if $\gcd(k,6)\left(\nu+\frac{1}{2}\right)\equiv \tfrac{k}{2}-\varrho-1 \pmod{6k},$} \\
		-1 & \text{if $\gcd(k,6)\left(\nu+\frac{1}{2}\right)\equiv \tfrac{k}{2}-\varrho+1 \pmod{6k}$,} \\
		0 & \text{if $\gcd(k,6)\left(\nu+\frac{1}{2}\right)\equiv \tfrac{k}{2}-\varrho \pmod{6k}$,}
	\end{cases}\\
	\label{E:deltadef}
	\delta_{k,\vp,\nu}&:=
	\begin{cases}
		1 & \text{if $\left|6kr+\gcd(k,6)\left(\nu+\tfrac{1}{2}\right)-\frac{k}{2}+\varrho\right|\leq 1$ for some $r\in \mathbb{Z}$,}\\
		0 & \text{otherwise.}
	\end{cases}
\end{align}

\begin{lemma}\label{lem:lem46}
	For $h=\mu\gcd(k,6)+\varrho$ with $0\le\varrho<\gcd(k,6)$, we have 
	\begin{multline*}
		\E^*_{\frac{[-H]_K}{K},\frac{1}{2\gcd(k,6)^2}}\left(\frac{\b}{iW}-\frac{1}{6K};\frac{[-H]_K}{K}+\frac{i}{KW}\right) = \frac{e^{-\frac{\pi i}{6K}}}{6ki}\sum_{\nu=0}^{6K-1} (-1)^{\nu+\mu}e^{\pi i\left(\frac{\mu-\nu}{3K}+\frac{[-H]_K}{K}\left(\nu-\mu+\frac12\right)^2\right)}\\
		\times \int_{-1}^1 \left(\cot\left(\tfrac{\pi}{6k}\left(x-\gcd(k,6)\left(\nu+\tfrac12\right)+\tfrac k2-\varrho+3kz\right)\right)-\frac{6k\d_{k,\vp,\nu}}{\pi(x+\g_{k,\vp,\nu}+3kz)}\right) e^{\frac{\pi}{6kZ}\left(1-x^2\right)} dx. 
	\end{multline*}
\end{lemma}

\begin{proof}
	Let $m=6Kr+\nu-\mu+\frac12$ with $0\le\nu\le6K-1$ and $r\in\Z$. Then $-\mu-\a_{H,K}=-\frac K2+\frac{\varrho}{\gcd(k,6)}$, and we have
	\begin{multline}\label{E:EEstarint}
		\E_{\frac{[-H]_K}{K},D}^*\left(\frac{\b}{iW}-\frac{1}{6K};\frac{[-H]_K}{K}+\frac{i}{KW}\right) = \frac{e^{\frac{2\pi D}{KW}-\frac{\pi i}{6K}}}{\pi i}\sum_{\nu=0}^{6K-1} (-1)^{\nu+\mu}e^{\pi i\left(\frac{\mu-\nu}{3K}+\frac{[-H]_K}{K}\left(\nu-\mu+\frac12\right)^2\right)}\\
		\times \lim_{N\to\infty} \sum_{\substack{-N\le r\le N\\\left|6Kr+\nu+\frac12-\frac K2+\frac{\varrho}{\gcd(k,6)}\right| > \frac{1}{\gcd(k,6)}}} \int_{-\sqrt{2D}}^{\sqrt{2D}} \frac{e^{-\frac{\pi x^2}{KW}}}{x-\left(6Kr+\nu+\frac12-\frac K2+\frac{\varrho}{\gcd(k,6)}-3Kz\right)} dx.
	\end{multline}
	Now \Cref{L:alphabadms} implies that $|6Kr+\nu+\frac12-\frac K2+\frac{\varrho}{\gcd(k,6)}|\le\frac{1}{\gcd(k,6)}$ if and only if $\d_{k,\varrho,\nu}=1$ and
	\[
		6Kr + \nu + \frac12 - \frac K2 + \frac{\varrho}{\gcd(k,6)} = -\frac{\g_{\varrho,k,\nu}}{\gcd(k,6)}.
	\]
	Thus, using $\pi \cot(\pi x)=\lim_{N \to \infty} \sum_{-N \leq r \leq N} \frac{1}{x+r}$, we have 
	\begin{multline*}
		\lim_{N\to\infty}\sum_{\substack{-N\le r\le N\\\left|6Kr+\nu+\frac12-\frac K2+\frac{\varrho}{\gcd(k,6)}\right| > \frac{1}{\gcd(k,6)}}} \frac{1}{x-\left(6Kr+\nu+\frac12-\frac K2+\frac{\varrho}{\gcd(k,6)}-3Kz\right)}\\
		= \frac{\pi}{6K}\left(\cot\left(\frac{\pi}{6K}\left(x-\nu-\frac12+\frac K2-\frac{\varrho}{\gcd(k,6)}+3Kz\right)\right) - \frac{6K\d_{k,\vp,\nu}}{\pi\left(x+\frac{\g_{k,\vp,\nu}}{\gcd(k,6)}+3Kz\right)}\right).
	\end{multline*}
	Note that if $\d_{k,\vp,\nu}=1$, then the above function has a removable singularities corresponding to the denominators in the right three terms. Plugging this into \eqref{E:EEstarint} and setting $D=\frac{1}{2\gcd(k,6)^2}$, we have
	\begin{multline*}
		\frac{e^{-\frac{\pi i}{6K}}}{6K i}\sum_{\nu=0}^{6K-1} (-1)^{\nu+\mu} e^{\pi i\left(\frac{\mu-\nu}{3K}+\frac{[-H]_K}{K}\left(\nu-\mu+\frac12\right)^2\right)} \int_{-\frac{1}{\gcd(k,6)}}^\frac{1}{\gcd(k,6)} e^{\frac{\pi}{K\gcd(k,6)^2W}-\frac{\pi x^2}{KW}}\\
		\times \left(\cot\left(\tfrac{\pi}{6K}\left(x-\nu-\tfrac12+\tfrac K2-\tfrac{\vp}{\gcd(k,6)}+3Kz\right)\right) -   \frac{6K\d_{k,\vp,\nu}}{\pi\left(x+\frac{\g_{k,\vp,\nu}}{\gcd(k,6)}+3Kz\right)}\right) dx.
	\end{multline*}
	Making the change of variables $x\mapsto\frac{x}{\gcd(k,6)}$ gives the claim. \qedhere
\end{proof}

\section{Taylor coefficients}\label{sec:taylorcoeff}

\subsection{Unrestricted unimodal sequences}

The following theorem determines the main term of derivatives (in $z$) of $\calU_1$.

\begin{theorem}\label{T:unrestrictedtaylormainterm}
	Let $\ell\in2\N_0$. For $Z\in\C$ with $\re(\frac1Z)\ge\frac k2$ and $|Z|\ll\frac1k$, we have
	\begin{multline*}
		\left[\frac{\del^\ell}{\del z^\ell} \calU_1\left(z;\frac hk+\frac{iZ}{k}\right)\right]_{z=0} = \frac{(2\pi i)^\ell i^\frac32}{4\sqrt{6}k}\chi\mat{[-h]_k&-\frac{h[-h]_k +1}{k}\\k&-h}^{-1}\\
		\times \sum_{\nu=0}^{2K-1} (-1)^\nu e^{\frac{\pi i[-h]_k}{12k}(12\nu(\nu+1)+1)}\sum_{j=0}^\frac\ell2 \binom{\ell}{2j}\sum_{\substack{a,b,c\ge0\\a+b+c=j}} k^a\k(a,b,c)Z^{\frac12-a-2c}\\
		\times \int_{-1}^1 C_{\ell-2j}\left(\frac{1}{2k}\left(\frac{x}{\sqrt{6}}-\nu-\frac12\right)\right) e^{\frac{\pi}{6kZ}\left(1-x^2\right)} dx + O\left(\log(k)|Z|^{\frac12-\ell}\right).
	\end{multline*}
\end{theorem}

\begin{proof}
	We first write  
	\begin{equation}\label{E:Cstaretaprinciplepart}
		\frac{\calC^*(z;\t)}{\eta(\t)} = q^{-\frac{1}{12}}\left(1+\sum_{j\ge0} \g_j(\t)\frac{(2\pi iz)^j}{j!}\right),
	\end{equation}
	where $|\g_j(\t)|\ll e^{2\pi i \tau}$ as $\t\to i\infty$. Recall the transformation of $\calU_1$ stated in Theorem \ref{thm:u1}. We first bound the Taylor coefficients in the first term (without prefactors). Define 
	\[
		\calU_1(z;\t) =: q^\frac{1}{24}\sum_{j\ge0} a_j(\t)\frac{(2\pi i z)^j}{j!}.
	\]
	We write
	\begin{align*}
		&f_{3k}\left(z;-Z\right) \calU_1\left(\frac{z}{iZ};\frac{[-h]_k}{k}+\frac{i}{kZ}\right)\\
		&\hspace{2.75cm}= e^{\frac{\pi i}{12}\left(\frac{[-h]_k}{k}+\frac{i}{kZ}\right)} \sum_{\ell\ge0} \left(\ell! \sum_{\substack{r,j\ge0\\j+2r=\ell}} \frac{b_r\left(3k;-Z\right)}{(2r)!j !} a_j\left(\frac{[-h]_k}{k}+\frac{i}{kZ}\right)\left(-\frac{1}{Z} \right)^j\right) \frac{(2\pi iz)^{\ell}}{\ell!}.
	\end{align*}
	Note that $a_j(\frac{[-h]_k}{k}+\frac{i}{kZ}) \ll 1$ as $Z \to 0$ with $\mathrm{Re}(Z)>0$, so 
	\begin{equation*}
		\left[\frac{\del^{\ell}}{\del z^{\ell}}\left( f_{3k}\left(z;-Z\right) \calU_1\left(\frac{iz}{Z};\frac{[-h]_k}{k}+\frac{i}{kZ}\right)\right) \right]_{z=0}\ll {|Z|}^{1-\ell} e^{-\frac{\pi}{12k}\re\left(\frac{1}{Z}\right)}.
	\end{equation*}
	
	The main term comes from the second term in Theorem \ref{thm:u1}.  Using the decomposition of $\calE$ and subtracting the principal part from $\frac{\calC^*(z;\t)}{\eta(\t)}$ in \eqref{E:Cstaretaprinciplepart}, the second term in Theorem \ref{thm:u1} equals 
	\begin{align}\nonumber
		&-\frac i2\chi\mat{[-h]_k&-\frac{h[-h]_k +1}{k}\\k&-h}^{-1}(iZ)^{-\frac12}f_k(z;Z)\\
		\nonumber
		&\hspace{1cm}\times \left({\vphantom{\frac{\left(\frac{[-h]_k}{k}\right)}{\left(\frac{[-h]_k}{k}\right)}}} e^{-\frac{\pi i[-h]_k}{6k}}\calE_{\frac{[-h]_k}{k},\frac{1}{12}}^*\left(\frac{2z}{iZ}; \frac{[-h]_k}{k}+\frac{i}{kZ}\right) + e^{-\frac{\pi i[-h]_k}{6k}} \calE_{\frac{[-h]_k}{k},\frac{1}{12}}^e\left(\frac{2z}{iZ}; \frac{[-h]_k}{k}+\frac{i}{kZ}\right)\right.\\
		\label{E:unrestrictedsecondtermdecomposition}
		&\hspace{2.7cm}\left.+\left( \frac{\calC^*\left(\frac{z}{iZ};\frac{[-h]_k}{k}+\frac{i}{kZ}\right)} {\eta\left(\frac{[-h]_k}{k}+\frac{i}{kZ}\right)} - e^{-\frac{\pi i}{6}\left(\frac{[-h]_k}{k}+\frac{i}{kZ}\right)}\right) \calE_{\frac{[-h]_k}{k},0}^e\left(\frac{2z}{iZ};\frac{[-h_k]}{k}+\frac{i}{kZ}\right)\right).
	\end{align}
	It is easy to show using Lemma \ref{lem:boundsfnuck} and Lemma \ref{lem:calEbound} that for $0\leq D\leq \frac{1}{12}$ we have 
	\begin{equation}\label{E:EeDbound}
		\left[\frac{\del^{\ell}}{\del z^{\ell}}\left( f_{k}(z;Z) \calE^e_{\frac{[-h]_k}{k},D}\left(\frac{2z}{iZ};\frac{[-h]_k}{k}+ \frac{i}{kZ}\right)\right)\right]_{z=0} \ll \log(k)|Z|^{1-\ell}.
	\end{equation}
	Combining \eqref{E:unrestrictedsecondtermdecomposition} and \eqref{E:EeDbound}, we have
	\begin{multline}\label{E:unrestrictedmaintermEstar}
		\left[\frac{\del^\ell}{\del z^\ell} \calU_1\left(z;\frac hk+\frac{iZ}{k}\right)\right]_{z=0} = -\frac i2\chi\mat{[-h]_k&-\frac{h[-h]_k +1}{k}\\k&-h}^{-1} (iZ)^{-\frac12}  e^{-\frac{\pi i[-h]_k}{6k}}\\
		\times \left[\frac{\del^\ell}{\del z^\ell}\left( f_k(z;Z) \calE_{\frac{[-h]_k}{k},\frac{1}{12}}^*\left(\frac{2z}{iZ};\frac{[-h]_k}{k}+\frac{i}{kZ}\right)\right)\right]_{z=0} + O\left(\log(k)|Z|^{\frac12-\ell}\right).
	\end{multline}
	By Lemma \ref{lem:calE*}, we have
	\begin{multline*}
		\left[\frac{\del^{\ell}}{\del z^{\ell}}\left( f_k(z;Z)\calE_{\frac{[-h]_k}{k},\frac{1}{12}}^* \left(\frac{2z}{iZ}; \frac{[-h]_k}{k}+\frac{i}{kZ}\right)\right) \right]_{z=0} =-\frac{i}{2\sqrt{6}k}\sum_{\nu=0}^{2K-1} (-1)^\nu e^{\pi i\left(\nu+\frac12\right)^2\frac{[-h]_k}{k}}\\
		\times  \left[\frac{\del^\ell}{\del z^\ell} \left(f_k(z;Z)\int_{-1}^{1} \cot\left(\frac{\pi}{2k}\left(\frac{x}{\sqrt{6}}-\nu- \frac12+2kz\right)\right)\right)\right]_{z=0}e^{\frac{\pi }{6kZ}\left(1-x^2\right)}dx.
	\end{multline*}
	We next compute, using the product rule and the fact that $f_k$ is an even function,
	\begin{multline*}
		\left[\frac{\del^\ell}{\del z^\ell} \left(f_k(z;Z)\cot\left(\frac{\pi}{2k}\left(\frac{x}{\sqrt{6}}-\nu-\frac12+2kz\right)\right)\right)\right]_{z=0}\\
		= (2\pi i)^{\ell}\sum_{\substack{a,b,c\geq0\\0\le j\le\frac\ell2}} \binom{\ell}{2j}\sum_{\substack{a,b,c\geq0\\a+b+c=j}} k^a\k(a,b,c)Z^{1-a-2c}C_{\ell-2j}\left(\frac{1}{2k}\left(\frac{x}{\sqrt{6}}-\nu-\frac12\right)\right).
	\end{multline*}
	Combining with the prefactors in \eqref{E:unrestrictedmaintermEstar} completes the proof.  
\end{proof}

\subsection{Durfee unimodal sequences}

The following theorem determines the main terms of derivatives of $\calV_1$. Define the following two indicator functions $\e_{\vp,k}^+$ and $\e_{\vp,k}^-$,
\begin{equation}\label{E:epsilonpmdef}
	\e_{k,\vp}^\pm :=
	\begin{cases}
		1 & \text{if } \vp\equiv\pm1\Pmod{\gcd(k,3)} \ \text{and } \gcd(k,12) \in \{1,2,3,6\},\\
		0 & \text{otherwise.}
	\end{cases}
\end{equation}
We have the following for the derivatives of $\mathcal{V}_1$.

\begin{theorem}\label{T:durfeeTaylormainterm}
	Let $\ell\in2\N_0$ and suppose $h=\gcd(k,6)\mu+\vp$ with $0\le\vp<\gcd(k,6)$. For $Z\in\C$ with $\re(\frac1Z)\ge\frac k2$ and $|Z|\ll\frac1k$, we have 
	\begin{multline*}
		\left[\frac{\del^\ell}{\del z^\ell} \calV_1\left(z;\frac hk+\frac{iZ}{k}\right)\right]_{z=0} = (2\pi i)^\ell\frac{i^\frac12e^{\frac{\pi i}{6k}\left(k-h-[-h]_k-\gcd(k,6)\right)}\sqrt{\gcd(k,6)}}{6\sqrt{6}k}\chi_{h,k}\\
		\times \sum_{\nu=0}^{6K-1} (-1)^{\nu+\mu}e^{\pi i\left(\frac{\mu-\nu}{3K}+\frac{[-H]_K}{K}\left(\nu-\mu+\frac12\right)^2\right)} \sum_{j=0}^\frac\ell2 \binom{\ell}{2j}\sum_{\substack{a,b,c\ge0\\a+b+c=j}} k^a\k(a,b,c)Z^{\frac12-a-2c}2^{2j-\ell}\int_{-1}^1 e^{\frac{\pi}{6kZ}\left(1-x^2\right)}\\
		\times \left(C_{\ell-2j}\left(\tfrac{1}{6k}\left(x-\gcd(k,6)\left(\nu+\tfrac12\right)+\tfrac k2-\varrho\right)\right) - \d_{k,\vp,\nu}(-1)^{\frac\ell2+j}2(\ell-2j)!\left(\tfrac{3k}{\pi (x+\g_{k,\vp,\nu})}\right)^{\ell-2j+1}\right) dx\\
		\\
		+ (2\pi i)^\ell i^{-\frac12} \sum_\pm \e_{k,\vp}^\pm (-1)^{\a_{H,K}\pm\frac{1}{\gcd(k,6)}+\frac12}e^{\pi i\left(-\frac{[-h]_k}{6k}+\frac{[-H]_K}{K}\left(\a_{H,K}\pm\frac{1}{\gcd(k,6)}\right)^2 + \frac{H}{36K}\mp\frac{1}{3\gcd(k,6)K}\right)} \hspace{-.12cm} \sqrt{\gcd(k,6)} \\
		\times \hspace{-.1cm} \chi_{h,k}  \hspace{-.1cm} \left(e^\frac{\pi}{6kZ} \hspace{-.2cm} \sum_{\substack{j_1,j_3\ge0,j_2\ge1\\2j_1+j_2+j_3=\ell}} \sum_{\substack{a,b,c\ge0\\a+b+c=j_1}} \sum_{0\le j_4\le\frac{j_2-1}{2}} \frac{\ell!3^{j_4}k^{a+\frac12+j_4}}{(2j_1)!j_2\cdot j_3!j_4!(j_2-1-2j_4)!2^{\ell+a-2j_1+j_4}\gcd(k,6)^{j_3}\pi^{j_4+1}}\right.\\
		\left.\times \k(a,b,c)(-1)^{a}(\pm i)^{j_2+1}(\mp i)^{j_3}Z^{1-a-2c-j_2-j_3+j_4}\right.\\
		\left.\mp \int_0^1 e^\frac{\pi t^2}{6kZ} dt \sum_{\substack{j_1,j_3\ge0\\2j_1+j_3=\ell}} \sum_{\substack{a,b,c\ge0\\a+b+c=j_1}} \binom{\ell}{2j_1} \k(a,b,c)\frac{k^{a-\frac12}(-1)^{a+1}}{3\cdot2^{a+j_3}}\left(\frac{\pm i}{\gcd(k,6)}\right)^{j_3}Z^{-a-2c-j_3}\right)\\
		+ O\left(\log(k)|Z|^{\frac12-\ell}\right).
	\end{multline*}
\end{theorem}

\begin{proof}
	Recall the transformation of $\calV_1$ in Lemma \ref{lem:calv}. We proceed as for Theorem \ref{T:unrestrictedtaylormainterm}. First consider the term with $\Psi$. We use \eqref{E:Cstaretaprinciplepart} and observe that the exponential growth is determined by
	\begin{multline*}
		e^{\frac{\pi}{6kZ}-\frac{\pi \alpha_{H,K}^2}{KW}}\Psi\left(\tfrac{\b}{iW}-\tfrac{1}{6K};\tfrac{[-H]_K}{K}+\tfrac{i}{KW}\right)\\ = ie^{\frac{\pi}{6kZ}-\frac{\pi \alpha_{H,K}^2}{KW}}{{\sum}^{*}_{m\in\Z+\frac12}} \sgn\left( m-K\b\right) (-1)^{m-\frac12} e^{\pi im^2\left(\frac{[-H]_K}{K}+\frac{i}{KW}\right)+2\pi im\left(\frac{\b}{iW}-\frac{1}{6K}\right)}.
	\end{multline*}
	As a function of $W$, exponential growth of the $m$-th term above is $e^{\frac{\pi}{KW}(\frac{1}{\gcd(k,6)^2}-(\a_{H,K}-m)^2)}$, which gives exponential decay for  $|m-\alpha_{H,K}|> \frac{1}{\gcd(k,6)}$. It follows that after differentiating $\ell$ times with respect to $z$ and setting $z=0$, the term with $\Psi$ in Lemma \ref{lem:calv}, goes into the error.
	
	Using \eqref{E:wtcalEDdefs}, \Cref{lem:seclem}, subtracting the principal part from $\frac{\calC^*(z;\t)}{\eta(\t)}$, and noting that $\frac{1}{6kZ}=\frac{1}{\gcd(k,6)^2KW}$, we write the term in \Cref{lem:calv} with $\E$ as
	\begin{align}\nonumber
		&-\chi_{h,k} (iW)^{-\frac12} e^{-\frac{\pi ih}{6k}+\frac{\pi i}{6}} f_k(z;Z) \left({\vphantom{\frac{\left(\frac{[-h]_k}{k}\right)}{\left(\frac{[-h]_k}{k}\right)}}} e^{-\frac{\pi i[-h]_k}{6k}}\E_{\frac{[-H]_K}{K},\frac{1}{2\gcd(k,6)^2}}^* \left(\frac{\b}{iW}-\frac{1}{6K};\frac{[-H]_K}{K}+\frac{i}{KW}\right)\right.\\
		\nonumber
		&\hspace{6cm}+ e^{-\frac{\pi i[-h]_k}{6k}}\E_{\frac{[-H]_K}{K},\frac{1}{2\gcd(k,6)^2}}^e \left(\frac{\b}{iW}-\frac{1}{6K};\frac{[-H]_K}{K}+\frac{i}{KW}\right)\\
		\label{E:durfeecalEfulltaylor}
		&\hspace{0.6cm}\left.+ {\vphantom{\frac{\left(\frac{[-h]_k}{k}\right)}{\left(\frac{[-h]_k}{k}\right)}}} \left(\frac{\calC^*\left(\frac{z}{iZ}; \frac{[-h]_k}{k}+\frac{i}{kZ}\right)}{\eta\left(\frac{[-h]_k}{k}+\frac{i}{kZ}\right)} - e^{-\frac{\pi i}{6}\left(\frac{[-h]_k}{k}+\frac{i}{kZ}\right)}\right) \E_{\frac{[-H]_K}{K},0}^e\left(\frac{\b}{iW}-\frac{1}{6K}; \frac{[-H]_K}{K}+\frac{i}{KW}\right)\right).
	\end{align}
	It is easy to show using Lemmas \ref{lem:boundsfnuck} and \ref{lem:del} that for $0\le D\le\frac{1}{2\gcd(k,6)^2}$ we have
	\[
		\left[\frac{\del^\ell}{\del z^\ell} \left(f_k(z;Z)\mathbb{E}_{\frac{[-H]_K}{K},D}^e\left(\frac{\b}{iW} - \frac{1}{6K};\frac{[-H]_K}{K}+\frac{i}{KW}\right)\right)\right]_{z=0} \ll \log(k)|Z|^{1-\ell}.
	\]
	We now use Lemma \ref{lem:lem46} and the above to write \eqref{E:durfeecalEfulltaylor} as
	\begin{align}\nonumber
		&-\chi_{h,k} (iW)^{-\frac12} e^{-\frac{\pi ih}{6k}+\frac{\pi i}{6}-\frac{\pi i[-h]_k}{6k}}\left[\frac{\del^\ell}{\del z^\ell} \left(f_k(z;Z)  \E_{\frac{[-H]_K}{K},\frac{1}{2\gcd(k,6)^2}}^*\left(\frac{\b}{iW}-\frac{1}{6K}; \frac{[-H]_K}{K} +\frac{i}{KW}\right)\right)\right]_{z=0} \nonumber  \\
		\nonumber
		&\hspace{10cm}+ O\left(\log(k)|Z|^{\frac12-\ell}\right)\\
		\nonumber
		&= i\chi_{h,k}(iW)^{-\frac12} \frac{e^{-\frac{\pi ih}{6k}+\frac{\pi i}{6}-\frac{\pi i[-h]_k}{6k}-\frac{\pi i}{6K}}}{6k}\sum_{\nu=0}^{6K-1} (-1)^{\nu+\mu} e^{\pi i\left(-\frac{\nu-\mu}{3K}+\frac{[-H]_K\left(\nu-\mu+\frac12\right)^2}{K}\right)} \int_{-1}^1 e^{\frac{\pi}{6kZ}\left(1-x^2\right)}\\
		\nonumber
		&\times \left[\pd{\ell}{z}\left({\vphantom{\frac{6k}{6K}}}f_k(z;Z) \left(\cot\left(\tfrac{\pi}{6k}\left(x-\gcd(k,6)\left(\nu+\tfrac12\right)+\tfrac k2-\vp+3kz\right)\right)-\frac{6k\d_{k,\vp,\nu}}{\pi(x+\g_{k,\vp,\nu}+3kz)}\right)\right)\right]_{z=0} dx\\
		\label{E:durfeetaylorfirstterm}
		&\hspace{10cm}+ O\left(\log(k)|Z|^{\frac12-\ell}\right).
	\end{align}
	By the product formula, we have
	\begin{multline*}
		\left[\tfrac{\del^\ell}{\del z^\ell}\left(f_k(z;Z) \left(\cot\left(\tfrac{\pi}{6k}\left(x-\gcd(k,6)\left(\nu+\tfrac12\right)+\tfrac k2-\vp+3kz\right)\right) - \frac{6k\d_{k,\vp,\nu}}{\pi(x+\g_{k,\vp,\nu}+3kz)}\right)\right)\right]_{z=0}\\
		= (2\pi i)^\ell\sum_{0\le j\le\frac\ell2} \binom{\ell}{2j}\sum_{\substack{a,b,c\ge0\\a+b+c=j}} k^a\k(a,b,c)Z^{1-a-2c}2^{2j-\ell}\\
		\times \left(C_{\ell-2j}\left(\tfrac{1}{6k}\left(x-\gcd(k,6)\left(\nu+\tfrac12\right)+\tfrac k2-\vp\right)\right)- \d_{k,\vp,\nu}(-1)^{\frac\ell2-j}2(\ell-2j)! \left(\tfrac{3k}{\pi(x+\g_{k,\vp,\nu})}\right)^{\ell-2j+1}\right),
	\end{multline*}
	which, if substituted into \eqref{E:durfeetaylorfirstterm} and then \eqref{E:durfeecalEfulltaylor}, is the first term in Theorem \ref{T:durfeeTaylormainterm}.
	
	The remaining terms in Lemma \ref{lem:calv} are 
	\begin{multline*}
		\chi_{h,k}(iW)^{-\frac12}f_\frac k2(z;-Z)\left(e^{-\frac{\pi i}{6}\left(\frac{[-h]_k}{k}+\frac{i}{kZ}\right)} + \left( \frac{\calC^*\left(\frac{z}{iZ};\frac{[-H]_K}{K}+\frac{i}{kZ}\right)}{\eta\left(\frac{[-H]_K}{K}+\frac{i}{kZ}\right)} - e^{-\frac{\pi i}{6}\left(\frac{[-h]_k}{k}+\frac{i}{kZ}\right)}\right)\right)\\
		\times {\sum}_{m\in\Z+\frac12}^{**} (-1)^{m+\frac12}\erf\left(i(3Kz+\a_{H,K}-m)\sqrt{\frac{\pi}{KW}}\right)\\
		\times e^{\frac{\pi z}{Z}(m-\a_{H,K})-\frac{\pi}{KW}(m-\a_{H,K})^2+\pi i\left(\frac{[-H]_Km^2}{K}+\frac{H}{36K}-\frac{1}{3K}(m-\a_{H,K})\right)}.
	\end{multline*}
	Recall that the summands with $|m-\a_{H,K}|\le\frac{1}{\gcd(k,6)}$ are described by Lemma \ref{L:alphabadms} as $m\in\{\a_{H,K},\,\a_{H,K}\pm\frac{1}{\gcd(k,6)}\}$. As in the previous subsection, the term with $e^{-\frac{\pi i}{6}(\frac{[-h]_k}{k}+\frac{i}{kZ})}$ dominates while the rest exponentially decays. 
	
	In the case that $m=\a_{H,K}$, we have that $f_\frac k2(z;-Z)\erf(3Kiz\sqrt{\frac{\pi}{KW}})$ is an odd function of $z$, so does not contribute to the $\ell$-th Taylor coefficient.
	
	In the case that $m=\a_{H,K}\pm\frac{1}{\gcd(k,6)}$, which by Lemma \ref{L:alphabadms} occurs if $\gcd(k,12)\in\{1,2,3,6\}$, we simplify using $\gcd(k,6)^2KW=6kZ$ to get
	\begin{multline}\label{E:erfterms}
		\chi_{h,k} (iW)^{-\frac12}f_{\frac{k}{2}}(z;-Z)  (-1)^{\alpha_{H,K}\pm\frac{1}{\gcd(k,6)}+\frac12} \erf\left(i\left(3Kz\mp\frac{1}{\gcd(k,6)}\right)\sqrt{\frac{\pi }{KW}}\right)   \\
		\times e^{\pm\frac{\pi z}{\gcd(k,6)Z}+\pi i\left(-\frac{[-h]_k}{6k}+\frac{[-H]_K}{K}\left( \alpha_{H,K}\pm\frac{1}{\gcd(k,6)}\right)^2+\frac{H}{36K}\mp\frac{1}{3\gcd(k,6)K}\right)}.
	\end{multline}
	Recalling the Taylor series of $f_\frac k2(z;-Z)$ in \eqref{E:f_kdef} and using
	\begin{multline*}
		\left[\frac{\del^\ell}{\del z^\ell} \erf\left(i\left(3Kz\mp\frac{1}{\gcd(k,6)}\right)\sqrt{\frac{\pi}{KW}}\right)\right]_{z=0}\\
		=
		\begin{cases}
			\mp\frac{\sqrt{2}i}{\sqrt{3kZ}}\int_0^1 e^\frac{\pi t^2}{6kZ} dt & \text{if }\ell=0,\\
			\vspace{-.4cm}\\
			\frac{(\mp\pi)^{\ell-1}\sqrt{6k}i}{Z^{\ell-\frac12}}e^\frac{\pi}{6kZ}\sum_{0\le j\le\frac{\ell-1}{2}} \frac{(\ell-1)!}{j!(\ell-1-2j)!}\left(\frac{3kZ}{2\pi}\right)^j & \text{if }\ell\ge1,
		\end{cases}
	\end{multline*}
	we see that the $\ell$-th Taylor coefficient in $z$ of \eqref{E:erfterms} is
	\begin{multline*}
		(2\pi i)^\ell i^{-\frac12}(-1)^{\a_{H,K}\pm\frac{1}{\gcd(k,6)}+\frac12}e^{\pi i\left(-\frac{[-h]_k}{6k}+\frac{[-H]_K}{K}\left(\a_{H,K}\pm\frac{1}{\gcd(k,6)}\right)^2 + \frac{H}{36K}\mp\frac{1}{3\gcd(k,6)K}\right)}\sqrt{\gcd(k,6)}\chi_{h,k}\\
		\times \left(e^\frac{\pi}{6kZ}\sum_{\substack{j_1,j_3\ge0,j_2\ge1\\2j_1+j_2+j_3=\ell}} \sum_{\substack{a,b,c\ge0\\a+b+c=j_1}} \sum_{0\le j_4\le\frac{j_2-1}{2}} \frac{\ell!3^{j_4}k^{a+j_4+\frac12}}{(2j_1)!j_2\cdot j_3!j_4!(j_2-1-2j_4)!2^{\ell+a-2j_1+j_4}\gcd(k,6)^{j_3}\pi^{j_4+1}}\right.\\
		\left.\times \k(a,b,c)(-1)^{a}(\pm i)^{j_2+1}(\mp i)^{j_3}Z^{1-a-2c-j_2-j_3+j_4}\right.\\
		\left.\mp \int_0^1 e^{\frac{\pi}{6kZ}t^2} dt \sum_{\substack{j_1,j_3\ge0\\2j_1+j_3=\ell}} \sum_{\substack{a,b,c\ge0\\a+b+c=j_1}} \binom{\ell}{2j_1} \k(a,b,c)\frac{k^{a-\frac12}(-1)^{1+a}}{3\cdot2^{a+j_3}}\left(\frac{\pm i}{\gcd(k,6)}\right)^{j_3}Z^{-a-2c-j_3}\right).
	\end{multline*}
	The theorem follows.
\end{proof}

\section{The Circle method and the Proofs of Theorems \ref{thm:ul} and \ref{thm:u2}}\label{sec:circmeth}

\subsection{Proof of Theorem \ref{thm:ul}}

We are now ready to prove Theorem \ref{thm:ul}.

\begin{proof}[Proof of Theorem \ref{thm:ul}]
	By Lemma \ref{L:genfnsrewrite} (1), we have 
	\[
		u_\ell(n) = \text{coeff}_{[q^n]}\frac{1}{(2\pi i)^\ell}\left[\frac{\del^\ell}{\del z^\ell} \left(q^{-\frac{1}{24}}\calU_1(z;\t)+H_1(\z;q)\right)\right]_{z=0}.
	\]
	The contribution of $H_1$ to $\sum_{m\in\Z}m^ku(m,n)$ can just be calculated directly. Indeed, if
	\[
		\frac{1}{(2\pi i)^\ell}\left[\frac{\del^\ell}{\del z^\ell} H_1(\z;q)\right]_{z=0} =: \sum_{n\ge0} a(n)q^n,
	\]
	then $a(n)\ll n^\frac{\ell-1}{2}$, so this piece goes into the error term. Write
	\[
		\calU_\ell(\t) := \frac{1}{(2\pi i)^{\ell}}\left[\frac{\partial^\ell}{\partial z^\ell} \calU_1(z;\t)\right]_{z=0} =: \sum_{n\ge0} a_\ell(n)q^{n+\frac{1}{24}}.
	\]
	With $Z=\frac kn-ik\TH$, $N=\flo{\sqrt{n}}$, $\vth_{h,k}':=\frac{1}{k(k_1+k)}$, and $\vth_{h,k}'':=\frac{1}{k(k_2+k)}$, where $\frac{h_1}{k_1}<\frac hk<\frac{h_2}{k_2}$ are three consecutive fractions in the Farey sequence of order $N$, we write
	\[
		a_\ell(n) = \sum_{\substack{0\le h<k\le N\\\gcd(h,k)=1}} e^{-\frac{2\pi i}{k}\left(n+\frac{1}{24}\right)h} \int_{-\vth_{h,k}'}^{\vth_{h,k}''} \calU_\ell\left(\frac hk+\frac{iZ}{k}\right)e^{\frac{2\pi\left(n+\frac{1}{24}\right)Z}{k}} d\TH.
	\]
	We now use Theorem \ref{T:unrestrictedtaylormainterm}.  The contribution of the error term may be bounded against
	\[
		\sum_{\substack{0\le h<k\le N\\\gcd(h,k)=1}} \log(k)\int_{-\vth_{h,k}'}^{\vth_{h,k}''} |Z|^{\frac12-\ell}e^{\frac{2\pi\left(n+\frac{1}{24}\right)\re(Z)}{k}} d\TH \ll \sum_{1\le k\le N} k\log(k)\int_{-\vth_{h,k}'}^{\vth_{h,k}''} |Z|^{\frac12-\ell}e^{2\pi\left(1+\frac{1}{24n}\right)} d\TH.
	\]
	Recall that $\vth_{h,k}',\vth_{h,k}''\asymp\frac{1}{kN}$. For $\ell=0$, we use that $|Z|^\frac12\ll(k|\TH|)^\frac12\ll n^{-\frac14}$ to bound the above against
	\[
		\ll n^{-\frac14}\frac1N\sum_{1\le k\le N} \log(k) \ll \log(N)n^{-\frac14} \ll \log(n)n^{-\frac14}.
	\]
	For $\ell\ge2$, we employ $|Z|^{\frac12-\ell}\le\frac{n^{\ell-\frac12}}{k^{\ell-\frac12}}$ and bound the error term as
	\[
		\ll \frac{n^{\ell-\frac12}}{N}\sum_{1\le k\le N} \frac{\log(k)}{k^{\ell-\frac12}} \ll \frac{n^{\ell-\frac12}}{N} \ll n^{\ell-\frac14} \ll n^{\frac{5\ell}{4}-\frac12}.
	\]

	Define the Kloostermann sums
	\begin{equation}\label{E:KloostermannUnrestricted}
		K_k(n,m) := i^\frac32(-1)^m\sum_{\substack{0\le h<k\\\gcd(h,k)=1}} \chi\mat{[-h]_k&-\frac{h[-h]_k+1}{k}\\k&-h}^{-1}e^{\frac{\pi i}{12k}(-(24n+1)h+(12m(m+1)+1)[-h]_k)}.
	\end{equation}
	Then the main term becomes
	\begin{multline*}\label{E:circlemethodalmostdone}
		\frac{1}{4\sqrt{6}}\sum_{\substack{1\le k\le N\\0\le\nu\le2k-1}} \frac{K_k(n,\nu)}{k}\sum_{0\le j\le\frac\ell2} \binom{\ell}{2j}\sum_{\substack{a,b,c\ge0\\a+b+c=j}} k^a\k(a,b,c)\\
		\times \int_{-1}^1 C_{\ell-2j}\left(\frac{1}{2k}\left(\frac{x}{\sqrt{6}}-\nu-\frac12\right)\right) \int_{-\vth_{h,k}'}^{\vth_{h,k}''} Z^{\frac12-a-2c} e^{\frac{\pi}{6kZ}\left(1-x^2\right)+\frac{2\pi\left(n+\frac{1}{24}\right)Z}{k}} d\TH dx.
	\end{multline*}
	The claim now follows using Lemma \ref{L:BesselBound} with $\nu=a+2c-\frac12$, $A=\frac{2\pi(n+\frac{1}{24})}{k}$, $B=\frac{\pi(1-x^2)}{6k}$, $\vth_1=\vth_{h,k}''$, and $\vth_2=\vth_{h,k}'$.
\end{proof}

\subsection{Proof of Theorem \ref{thm:u2}}

We are now ready to prove Theorem \ref{thm:u2}.

\begin{proof}[Proof of Theorem \ref{thm:u2}]
	By Lemma \ref{L:genfnsrewrite} (2), we have 
	\[
		v_\ell(n) = \coeff_{\left[q^n\right]}\frac{1}{(2\pi i)^\ell}\left[\frac{\del^\ell}{\del z^\ell} \left(q^{-\frac14}\calV_1(z;\t)+H_2(\z;q)\right)\right]_{z=0}.
	\]
	Again, $H_2(\z;q)$ only contributes to the error term. Write 
	\[
		\calV_\ell(\t) := \frac{1}{(2\pi i)^\ell}\left[\frac{\del^\ell}{\del z^\ell} \calV_1(z;\t)\right]_{z=0} =: \sum_{n\ge0} b_\ell(n)q^{n+\frac14}.
	\]
	With the same setup and notation as above, we have
	\begin{equation}\label{E:bellcircle}
		b_\ell(n) = \sum_{\substack{0\le h<k\le N\\\gcd(h,k)=1}} e^{-\frac{2\pi ih}{k}\left(n+\frac14\right)} \int_{-\vth_{h,k}'}^{\vth_{h,k}''} \calV_\ell\left(\frac hk+\frac{iZ}{k}\right)e^{\frac{2\pi\left(n+\frac14\right)Z}{k}} d\TH.
	\end{equation}
	Now set $h=\gcd(k,6)\mu+\vp$ with $0\le\vp<\gcd(k,6)$. Then we can rewrite the sum above as
  	\[
		\sum_{\substack{0\le h<k\le N\\\gcd(h,k)=1}} = \sum_{\substack{1\le k\le N\\0\le\vp\le\gcd(k,6)-1\\\gcd(\vp,\gcd(k,6))=1}} \sum_{\substack{0\le\mu\le K-1\\\gcd(k,h)=1}}.
	\]
	Continuing to write $h=\gcd(k,6)\mu+\vp$, we set
	\begin{align}\nonumber
		\calK_{k,\vp,\nu}(n) &:= i^\frac12 e^{\frac{\pi i}{6}}(-1)^{\nu}\sum_{\substack{0\le\mu\le K-1\\\gcd(k,h)=1}} (-1)^\mu\chi_{h,k}\\
		\label{E:DurfeeKloostermann}
		&\hspace{1.8cm}\times e^{\frac{\pi i}{6k}\left(-(12n+4)h-[-h]_k+\gcd(k,6)(2(\mu-\nu)-1)+6\gcd(k,6)\left(\nu-\mu+\tfrac12\right)^2[-H]_K\right)},\\
		\nonumber
		\calK_{k,\vp,j_2,j_3,a}^{\pm,[1]}(n) &:= i^\frac32e^{\mp\frac{\pi i}{3k}}(-1)^{a+\frac12\pm\frac{1}{\gcd(k,6)}}(\pm i)^{j_2+1}(\mp i)^{j_3}\sum_{\substack{0\le\mu\le K-1\\\gcd(k,h)=1}} (-1)^{\a_{H,K}}\chi_{h,k}\\
		\label{E:erfDurfeeKloostermann1}
		&\hspace{3cm}\times e^{\frac{\pi i}{6k}\left(-2(6n+1)h-[-h]_k+6\gcd(k,6)\left(\a_{H,K}\pm\frac{1}{\gcd(k,6)}\right)^2[-H]_K\right)},\\
		\nonumber
		\calK_{k,\vp,j_3,a}^{\pm,[2]}(n) &:= i^\frac32e^{\mp\frac{\pi i}{3k}}(-1)^{a+\frac12\pm\frac{1}{\gcd(k,6)}}(\pm i)^{j_3}\sum_{\substack{0\le\mu<K-1\\\gcd(k,h)=1}} (-1)^{\a_{H,K}}\chi_{h,k}\\
		\label{E:erfDurfeeKloostermann2}
		&\hspace{3cm}\times e^{\frac{\pi i}{6k}\left(-2(6n+1)h-[-h]_k+6\gcd(k,6)\left(\a_{H,K}\pm\frac{1}{\gcd(k,6)}\right)^2[-H]_K\right)}.
	\end{align}
	 We plug Theorem \ref{T:durfeeTaylormainterm} into \eqref{E:bellcircle} and conclude Theorem \ref{thm:u2} as in the previous subsection.
\end{proof}

\section{Outlook}\label{sec:outlook}

We demonstrate with two examples techniques for proving precise asymptotic series for Taylor coefficients of functions $F(\z;q)$ that factor as in \eqref{E:genfnfactor}. There are surely many more examples for which a similar analysis is possible. Notable among them are ranks for strongly unimodal sequences (in which the inequalities \eqref{E:unimodaldef} are strict) and ranks for overpartitions. The relevant generating functions are essentially mock Jacobi forms and are, respectively, \cite[equation (2.1)]{BJSMR} given by
\begin{align*}
	U^*(\zeta;q)=-\frac{1}{(1+\zeta^{-1})(q;q)_{\infty}}\left(\sum_{n \in \mathbb{Z} \setminus \{0\}} \frac{(-1)^nq^{\frac{n(3n+1)}{2}}}{1+\zeta q^n}+\sum_{n \in \mathbb{Z} \setminus \{0\}} \frac{q^{\frac{n(n+1)}{2}}\zeta^{-n}}{1+\zeta q^n}\right),
\end{align*}
and \cite[equation (1.3)]{BringmannLovejoy}
\begin{align*}
	\mathcal{O}(\zeta;q)=\frac{(-q;q)_{\infty}}{(q;q)_{\infty}}\left(1+2\sum_{n \geq 1} \frac{(-1)^n(1-\zeta)\left(1-\zeta^{-1}\right)q^{n^2+n}}{(1-\zeta q^n)\left(1-\zeta^{-1}q^n\right)} \right).
\end{align*}
Other examples depending on false theta functions include the companions to Caparelli's identities found by the second author and Mahlburg \cite[Theorem 1.1]{KMCaparelli}.

Finally, let $\calU(n)$ be the set of unimodal sequences with $u(n):=\#\calU(n)$, and define $\calV(n)$ and $v(n)$ for Durfee unimodal sequences similarly. The second author, Jennings-Shaffer, and Mahlburg used the method of moments to conclude limiting logistic distributions \cite[Proposition 1.2]{BJSM},
$$
	\lim_{n \to \infty} \frac{\#\left\{\sigma \in \mathcal{U}(n) : \frac{\mathrm{rank}(\sigma)}{\sqrt{3n}}\leq x \right\}}{u(n)}=\lim_{n \to \infty} \frac{\#\left\{\sigma \in \mathcal{V}(n) : \frac{\mathrm{rank}(\sigma)}{\sqrt{3n}}\leq x \right\}}{v(n)}=\frac{1}{1+e^{-\pi x}}.
$$
We leave it as an open problem to determine convergence rates above.

\end{document}